\documentclass{psapm-l}

\newtheorem{theorem}{Theorem}[section]

\usepackage{lipsum}
\usepackage{graphicx}
\usepackage{amsmath}
\usepackage{subfigure}
\usepackage{amsthm}
\usepackage{amsfonts}
\usepackage{amssymb}
\usepackage{graphicx}
\usepackage{caption}
\usepackage{verbatim}
\usepackage{wrapfig}
\usepackage{color}
\usepackage{float}
\usepackage{algpseudocode}
\usepackage{algorithm}
\usepackage[active]{srcltx}
\usepackage[unicode]{hyperref}
\usepackage{enumerate}
\usepackage[margin=0.7in]{geometry}

\theoremstyle{definition}
\newtheorem{definition}[theorem]{Definition}
\newtheorem{example}[theorem]{Example}

\newtheorem{proposition}[theorem]{Proposition}
\newtheorem{corollary}[theorem]{Corollary}

\theoremstyle{remark}
\newtheorem{remark}[theorem]{Remark}

\numberwithin{equation}{section}

\begin{document}

\title{Engineering and Business Applications of Sum of Squares Polynomials}

\author{Georgina Hall}
\address{Department of Decision Sciences, INSEAD, Fontainebleau, 77300, France}
\email{georgina.hall@insead.edu}


\subjclass[2000]{49-06,49-02,90C22,90C90}


\keywords{Sum of squares optimization, engineering applications, dynamical systems and control, moment problems, option pricing, shape-constrained regression, optimal design, polynomial games, copositive and robust semidefinite optimization}

\begin{abstract}
Optimizing over the cone of nonnegative polynomials, and its dual counterpart, optimizing over the space of moments that admit a representing measure, are fundamental problems that appear in many different applications from engineering and computational mathematics to business. In this paper, we review a number of these applications. These include, but are not limited to, problems in control (e.g., formal safety verification), finance (e.g., option pricing), statistics and machine learning (e.g., shape-constrained regression and optimal design), and game theory (e.g., Nash equilibria computation in polynomial games). We then show how sum of squares techniques can be used to tackle these problems, which are hard to solve in general. We conclude by highlighting some directions that could be pursued to further disseminate sum of squares techniques within more applied fields. Among other things, we briefly address the current challenge that scalability represents for optimization problems that involve sum of squares polynomials and discuss recent trends in software development.
\end{abstract}

\maketitle

A variety of applications in engineering, computational mathematics, and business can be cast as optimization problems over the cone of nonnegative polynomials or the cone of moments admitting a representing measure. For a long while, these problems were thought to be intractable until the advent, in the 2000s, of techniques based on sum of squares (sos) optimization. The goal of this paper is to provide concrete examples---from a wide variety of fields---of settings where these techniques can be used and detailed explanations as to how to use them. The paper is divided into four parts, each corresponding to a broad area of interest: Part 1 covers control and dynamical systems, Part 2 covers probability and measure theory, Part 3 covers statistics and machine learning, and Part 4 covers optimization and game theory. Each part is further subdivided into sections, which correspond to specific problems within the broader area such as, for Part 1, certifying properties of a polynomial dynamical system. Each section is purposefully written so that limited knowledge of the application field is needed. Consequently, a large part of each section is spent on providing a mathematical framework for the field and couching the question of interest as an optimization problem involving nonnegative polynomials and/or moment constraints. A shorter part explains how to go from these (intractable) problems to (computationally-tractable) sos programs. The conclusion of this paper briefly touches upon some implementation challenges faced by sum of squares optimization and the subsequent research effort developed to counter them.

\part{Dynamical systems and control}\index{dynamical systems} \index{control}

A dynamical system is a system whose state varies over time. Broadly speaking, a \emph{state} is a vector $x(t) \in \mathbb{R}^n$ that provides enough information on the system at time $t$ for one to predict future values of the state if the system is left to its own devices. For example, if we consider a physical system such as a rolling ball, then one could, e.g., consider the position and instantaneous velocity of its center as a 6-dimensional state vector. Or, if the problem at hand is a study of the evolution of the population of wolves and sheep in a certain geographical region, the state vector of a simple model could simply encompass the current number of wolves and sheep in that region.

As the state vector contains enough information that one can predict its evolution if there is no outside interference, it is possible to relate future states back to the current state via so-called \emph{state equations}. Their expression varies depending on whether the system is \emph{discrete time} or \emph{continuous time}. In a discrete-time system, the state $x(k)$ is defined for discrete times $k=0,1,2,\ldots$, and we have
\begin{align}\label{eq:discrete.time}
x(k+1)=f(x(k)),
\end{align}
where $f$ is some function from $\mathbb{R}^n$ to $\mathbb{R}^n$. In a continuous-time system, the state $x(t)$ varies continuously with time $t\geq 0$ and we have
\begin{align}\label{eq:cont.time}
\frac{d x(t)}{dt}=f(x(t)),
\end{align}
where $f$ is again some function from $\mathbb{R}^n$ to $\mathbb{R}^n$. The goal is generally to understand how the \emph{trajectory} $\{x(k)\}_k$, solution to (\ref{eq:discrete.time}), or $t \mapsto x(t)$, solution to (\ref{eq:cont.time}), behaves over time. Sometimes such solutions can be computed explicitly and then it is easy to infer their behavior: this is the case for example when $f$ is linear, that is, when $f(x)=Ax$ for some $n \times n$ matrix $A$; see, e.g., \cite{linear}. However, when $f$ is more complex, computing closed-form solutions to (\ref{eq:discrete.time}) or (\ref{eq:cont.time}) can be hard, even impossible, to do. The goal is then to get insights as to different properties of the trajectories without ever having to explicitly compute them. For example, it may be enough to know that the ball we were considering earlier avoids a certain puddle, or that our wolf population always stays within a certain range. This is where sum of squares polynomials come into play---as algebraic certificates of properties of dynamical systems. In Sections \ref{sec:stab} and \ref{sec:col.avoid}, we will see how we can certify \emph{stability} and \emph{collision avoidance} of polynomial dynamical systems (i.e., dynamical systems as in (\ref{eq:discrete.time}) and (\ref{eq:cont.time}) where $f$ is a polynomial) using sum of squares. Other properties such as ``invariance'' or ``reachibility'' can be certified using sum of squares as well but are not covered here.

We will also review more complex models than what is given in (\ref{eq:discrete.time}) and (\ref{eq:cont.time}). For example, we have assumed here that our dynamical system is \emph{autonomous}. This means that the function $f$ only depends on $x(t)$ or $x(k)$. But this need not be the case. The function $f$ could also depend on, say, an external input $u(t) \in \mathbb{R}^p$. This is a well-studied class of dynamical systems and the vector $u(t)$ is termed a \emph{control}. We will briefly touch upon an example of such a system  in Section \ref{sec:stab}. Another alternative to (\ref{eq:discrete.time}) and (\ref{eq:cont.time}) could be a direct dependency of $f$ on time on top of its dependency on $x(t)$ or $x(k)$. In this case, such a system is called \emph{time-varying}. We will see an example of such a system in Section \ref{sec:switched}. In its most general setting, $f$ can be a function of all three: time, state, and control, but we do not cover problems of this type in their full generality here; see, e.g., \cite{Khalil:3rd.Ed} for more information.

\section{Certifying properties of a polynomial dynamical system} \index{polynomial dynamical system} \index{Lyapunov functions}
In this section, we consider a continuous-time polynomial dynamical system:
\begin{align} \label{eq:ode.stability}
\dot{x}=f(x),
\end{align}
where $\dot{x}$ is the derivative of $x(t)$ with respect to $t$ and $f:\mathbb{R}^n \rightarrow \mathbb{R}^n$ is a vector field, every component of which is a (multivariate) polynomial.

\subsection{Stability}\label{sec:stab} \index{Stability of a polynomial dynamical system}

Let $\bar{x}$ be an equilibrium point of (\ref{eq:ode.stability}), that is a point such that $f(\bar{x})=0$. Note that by virtue of the definition, any system that is initialized at its equilibrium point will remain there indefinitely. For convenience, we will assume that the equilibrium point is the origin. This is without loss of generality as we can always bring ourselves back to this case by performing a change of variables $y=x-\bar{x}$ in (\ref{eq:ode.stability}). Our goal is to study how the system behaves around its equilibrium point.

\begin{definition}
	The equilibrium point $\bar{x}=0$ of (\ref{eq:ode.stability}) is said to be \emph{stable} if, for every $\epsilon>0$, there exists $\delta(\epsilon)=\delta>0$ such that $$||x(0)||<\delta \Rightarrow ||x(t)||<\epsilon, ~\forall t\geq 0.$$
\end{definition}
This notion of stability is sometimes referred to as stability in the sense of Lyapunov, in honor of the Russian mathematician Aleksandr Lyapunov (1857-1918).

Intuitively, this notion of stability corresponds to what we would expect it to be: 
there always exists a ball around the equilibrium point from which trajectories can start with the guarantee that they will remain close to the equilibrium in the future, where the notion of ``close'' can be prescribed.

\begin{definition}
	The equilibrium point $\bar{x}=0$ of (\ref{eq:ode.stability}) is said to be \emph{locally asymptotically stable} if it is stable and if there exists $\delta'$ such that $$||x(0)||<\delta' \Rightarrow \lim_{t \rightarrow \infty} x(t)=0.$$ 
\end{definition}

\begin{definition}
	The equilibrium point $\bar{x}=0$ of (\ref{eq:ode.stability}) is said to be \emph{globally asymptotically stable} (GAS) if it is stable and if, $\forall x(0) \in \mathbb{R}^n$,  $\lim_{t \rightarrow \infty} x(t)=0.$
\end{definition}

We will focus on how one can certify \emph{global} asymptotic stability of an equilibrium point in the rest of this section. Analogous results to the ones discussed here exist for both stability and local asymptotic stability and can be found in \cite[Chapter 4]{Khalil:3rd.Ed}. The key notion that is used here is that of a \emph{Lyapunov function}, developed by Lyapunov in his thesis \cite{PhD:Lyapunov}. The theorem we give below appears in, e.g., \cite{Khalil:3rd.Ed}. It uses the notation $\nabla g(x)$ for the gradient of a function $g:\mathbb{R}^n \rightarrow \mathbb{R}$.
\begin{theorem}\label{th:lyap}
	Let $\bar{x}=0$ be an equilibrium point for (\ref{eq:ode.stability}). If there exists a continuously differentiable function $V:\mathbb{R}^n \rightarrow \mathbb{R}$ such that
	\begin{enumerate}[(i)]
		\item $V$ is radially unbounded, i.e., $||x|| \rightarrow \infty \Rightarrow V(x) \rightarrow \infty$
		\item $V$ is positive definite, i.e., $V(x)>0, \forall x\neq 0$ and $V(0)=0$
		\item $\dot{V}(x)\mathrel{\mathop{:}}=\nabla V(x)^T f(x)<0$ for all $x\neq 0$ and $\dot{V}(0)=0$
	\end{enumerate}
	then $\bar{x}$ is globally asymptotically stable.
\end{theorem}
Such a function $V$ is called a Lyapunov function and can be viewed as the generalization of an energy function. The function $\dot{V}$ is the derivative of $V$ with respect to its trajectory as it is equal to $\frac{d}{dt} V(x(t))$ where $x(t)$ is a solution to (\ref{eq:ode.stability}). The proof of the theorem is omitted but can be found in \cite[Chapter 4]{Khalil:3rd.Ed}. 

This theorem states a sufficient condition for the equilibrium point to be GAS. Is it the case that whenever the system is GAS, such a Lyapunov function exists? These type of questions give rise to what is known as \emph{converse Lyapunov theorems}. The one given below comes from \cite{kurzweil1956inversion} but this precise formulation appears in \cite{Bacciotti.Rosier.Liapunov.Book}.

\begin{theorem}\label{th:converse}
	Let $f$ be continuous. If $\bar{x}=0$ is globally asymptotically stable for (\ref{eq:ode.stability}) then there exists an $\infty$-differentiable function $V:\mathbb{R}^{n}\rightarrow \mathbb{R}$ satisfying properties (i)-(iii) of Theorem \ref{th:lyap}.
\end{theorem}

Similar theorems to Theorem \ref{th:converse} exist for stability and local asymptotic stability; see \cite{Bacciotti.Rosier.Liapunov.Book}. Theorems such as these do not help us however to explicitly compute a Lyapunov function $V$, as they do not give a tractable construction of its existence. This is where sum of squares techniques come in useful as we see now. So the techniques can be directly applied, we restrict ourselves to polynomial Lyapunov functions of a certain degree. In practice, this does not seem too restrictive a choice in the context of polynomial dynamical systems and enables a finite parametrization of the Lyapunov functions by their coefficients.  A disadvantage however to such a restriction is that, while a $C^\infty$ Lyapunov function was guaranteed to exist for a polynomial dynamical system, a polynomial Lyapunov function is not. In fact, examples of polynomial dynamical systems that do not have polynomial Lyapunov functions exist as can be seen below.

\begin{theorem}\cite{AAA_MK_PP_CDC11_no_Poly_Lyap}
	Consider the polynomial vector field
	\begin{equation}\label{eq:CE_stab}
	\begin{aligned}
	\dot{x}&=-x+xy\\
	\dot{y}&=-y.
	\end{aligned}
	\end{equation}
	The origin is a globally asymptotically stable equilibrium point, but the system does not admit a polynomial Lyapunov function.
\end{theorem}

The proof of this theorem is omitted here but can be found in \cite{AAA_MK_PP_CDC11_no_Poly_Lyap}. The crux of it relies on showing global asymptotic stability of the origin by producing a (non polynomial) Lyapunov function $V(x,y)=\ln(1+x^2)+y^2$, and then showing that no polynomial Lyapunov function could exist due to the exponential growth rates of the trajectories (see Figure \ref{fig:no.poly.Lyap}). 
\begin{figure}[H]
	\centering
	\includegraphics[scale=0.25]{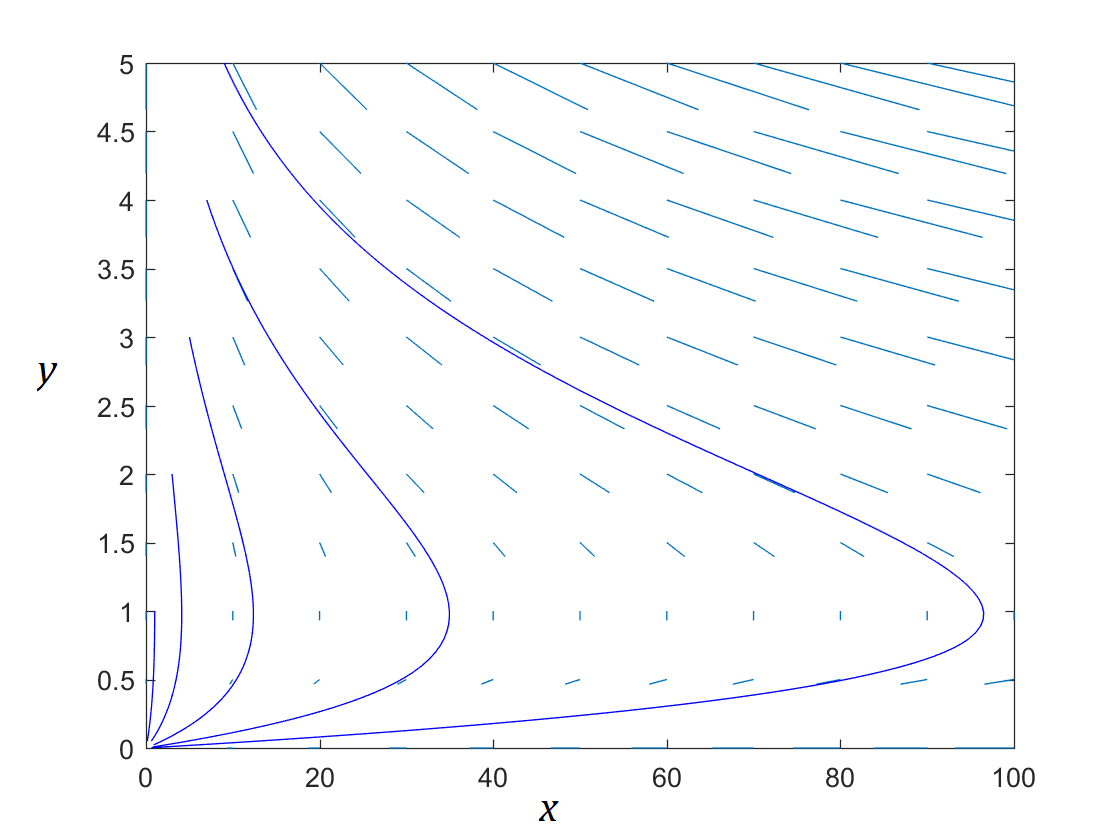}
	\caption{Representation of the polynomial vector field given in (\ref{eq:CE_stab}) with some trajectories}
	\label{fig:no.poly.Lyap}
\end{figure}

Though this result is negative in nature, it is worth noting that some positive results do exist. In particular, in \cite{Peet.exp.stability} it is shown that \emph{exponentially stable} polynomial dynamical systems always have polynomial Lyapunov functions on compact sets (we do not define exponential stability here but, at a high level, it is a stronger notion than asymptotic stability as it requires rates of convergence of trajectories to the equilibrium point rather than simply convergence).

Restricting ourselves to polynomials is a step in the right direction for making the problem of searching for Lyapunov functions computationally tractable. It is not enough however. Indeed, searching for polynomial Lyapunov functions that verify (i)-(iii) is a hard problem as it requires constraining polynomials to be positive on $\mathbb{R}^n$, a problem that we know to be hard for degree-4 polynomials already~\cite{nonnegativity_NP_hard}. As expected, this is where sum of squares polynomials come into play. A few references on the use of sum of squares optimization in showing asymptotic stability of a polynomial system include \cite{PhD:Parrilo,PositivePolyInControlBook,PapP02}. We present a condensed version of these references below.

\begin{definition}\label{def:sos.Lyap.fct}
	A polynomial function $V$ is a \emph{sum of squares Lyapunov function} for the polynomial system in (\ref{eq:ode.stability}) if it vanishes at the origin and satisfies
	\begin{enumerate}[(i')]
		\item $V$ is sos
		\item $-\dot{V}$ is sos.
	\end{enumerate}
\end{definition}
Note that as $V(0)=0$ and $V$ is sos, the constant and linear terms of $V$ must be zero. It is clear that requiring $V$ to be sos and $-\dot{V}$ to be sos implies that they will be nonnegative. This is not however what is required in Theorem~\ref{th:lyap}: there, $V$ and $-\dot{V}$ need to be positive definite. Furthermore, $V$ has to be radially unbounded. How can positive definiteness and radial unboundedness be enforced in practice? One suggestion to enforce positive definiteness of $V$ and $-\dot{V}$ is given in \cite[Proposition 5]{papachristodoulou2005tutorial}, that we repeat here.

\begin{proposition}
	Given a polynomial $V(x)$ of degree $2d$, let $\phi_{\epsilon}(x)=\sum_{i=1}^n \sum_{j=1}^d \epsilon_{ij}x_i^{2j}$ where $\epsilon_{ij}\geq 0$ for all $i$ and $j$ and $$\sum_{j=1}^d \epsilon_{ij}>\gamma, \text{ for all } i=1,\ldots,n,$$ when $\gamma$ is some fixed positive constant. If there exist some $\epsilon=(\epsilon_{ij})_{ij}$ verifying the previous conditions and if $V-\phi_{\epsilon}$ is sos, then it follows that $V$ is positive definite.
\end{proposition}

In the case where $V$ is taken to be a homogeneous\footnote{A function $f:\mathbb{R}^n \rightarrow \mathbb{R}$ is said to be homogeneous of degree $d$ if $f(\lambda x)=\lambda^{d} f(x)$, for any scalar $\lambda$.} polynomial of degree $2d$, then one need only keep the monomials of degree $2d$ in $\phi_{\epsilon}(x)$. In other words, we constrain $V(x)-\sum_{i=1}^n \epsilon_i x_i^{2d}$ to be sos. 

For radial unboundedness, it is well known that a polynomial $V$ is radially unbounded if its top homogeneous component, i.e., the homogeneous polynomial formed by the collection of the highest order monomials of $V$, is positive definite. This can be enforced as described in the paragraph above.
%
%

In practice however, as discussed in \cite[page 41]{AAA_MS_Thesis}, these conditions are unwieldy and can usually be done away with. Indeed, finding a polynomial $V$ that satisfies conditions (i')-(ii') is a sum of squares program with no objective function. When solving programs of this type with interior point methods, the solution returned is in the interior of the feasible set, and hence will not vanish other than at the origin. This implies that in general, one would obtain polynomials $V$ that satisfy conditions (i)-(iii). This should be checked numerically however. For (i)-(ii), this can be done by checking the eigenvalues of the Gram matrices associated to $V$ and to $-\dot{V}$; for (iii), this can be done by checking the eigenvalues of the Gram matrix associated to the top homogeneous component of $V$. Note that it is not enough to check (i)-(ii): (iii) needs to be checked too. Indeed, there exist\footnote{Thank you to Jeffrey Zhang for finding this example!} polynomials that are zero at zero, positive everywhere else, and not radially unbounded, such as:
\begin{align}\label{eq:CE.Jeff}
p(x,y)=x^2y^2+4x^2y+5x^2+xy^2+2xy+0.25y^2.
\end{align}

%
%
%
\begin{remark}
Being a sum of squares polynomial is a sufficient, but not necessary, condition for being nonnegative. This doesn't imply however that conditions (i')-(ii') are much more conservative than (ii)-(iii) for polynomial $V$. Indeed, there may be many polynomials satisfying conditions (ii)-(iii), some of which not having a sum of squares certificate, but as long as one of them does, then (i')-(ii') should not be more conservative than (ii)-(iii) (with the technical details considered above in mind). 
This motivates the study of converse questions around the existence of \emph{sum of squares Lyapunov functions} if a polynomial Lyapunov function is known to exist. It is known that if a polynomial Lyapunov function of degree $2d$ exists, it does not follow that an sos Lyapunov function of degree $2d$ exists; see an example in \cite[Section 3.1]{AAA_PP_CDC11_converseSOS_Lyap}. The related question as to whether an sos Lyapunov function of higher degree exists if a polynomial Lyapunov function exists is open for general polynomial dynamical systems. When we restrict ourselves to homogeneous polynomial dynamical systems (i.e., when $f$ is homogeneous) the latter conclusion can be made, see \cite{AAA_PP_CDC11_converseSOS_Lyap}.
\end{remark}

\begin{example}
	As an illustrative example of what we have seen so far, we consider a model of a jet engine given in \cite{krstic1995nonlinear} and revisited in \cite{PabloGregRekha_BOOK}. The dynamics of the engine are given by
	\begin{equation}\label{eq:jet}
	\begin{aligned}
	\dot{x}&=-y-\frac{3}{2}x^2-\frac12 x^3\\
	\dot{y}&=3x-y.
	\end{aligned}
	\end{equation}
	We wish to show that the origin is globally asymptotically stable. Using MATLAB and the software package YALMIP\cite{yalmip}, we search for a polynomial Lyapunov function $V$ for this system satisfying (i') and (ii'). We start by capping the degree of $V$ at 2, then 4. The solver returns $V=0$ for degree 2 but a nonzero solution for degree 4. It is easy to check numerically that $V$ is positive definite and radially unbounded, and that $-\dot{V}$ is positive definite too. Hence, the origin is GAS for (\ref{eq:jet}). The vector field as well as trajectories of the system and level sets of $V$ are plotted in Figure \ref{fig:example.stab}.
\end{example}

\subsubsection{The specific case of linear systems.} In the particular case where the dynamical system is linear, that is 
\begin{align}\label{eq:lin.syst}
\dot{x}=Ax
\end{align}
where $A$ is an $n \times n$ matrix, the previous results simplify considerably. Indeed, $\bar{x}=0$ is a GAS equilibrium point for (\ref{eq:lin.syst}) if and only if a quadratic Lyapunov function $V$ exists. As $V$ is quadratic, it can be parametrized as $V(x)=x^TPx$ where $P$ is a symmetric $n \times n$ matrix. Enforcing conditions (i)-(iii) then simply amounts to searching for a matrix $P$ such that $$P \succ 0 \text{ and } A^TP+PA \succ 0.$$ The search for $P$ is a semidefinite program. If $\bar{x}$ is GAS then such a system will be feasible (see \cite[Chapter 5]{BoydLmiBook} and \cite[Section 2.2]{PabloGregRekha_BOOK} for the discrete-time case). In the language of sum of squares, one can say that GAS linear systems always admit a quadratic sum of squares Lyapunov function.

\begin{figure}[H]
	\centering
	\includegraphics[scale=0.25]{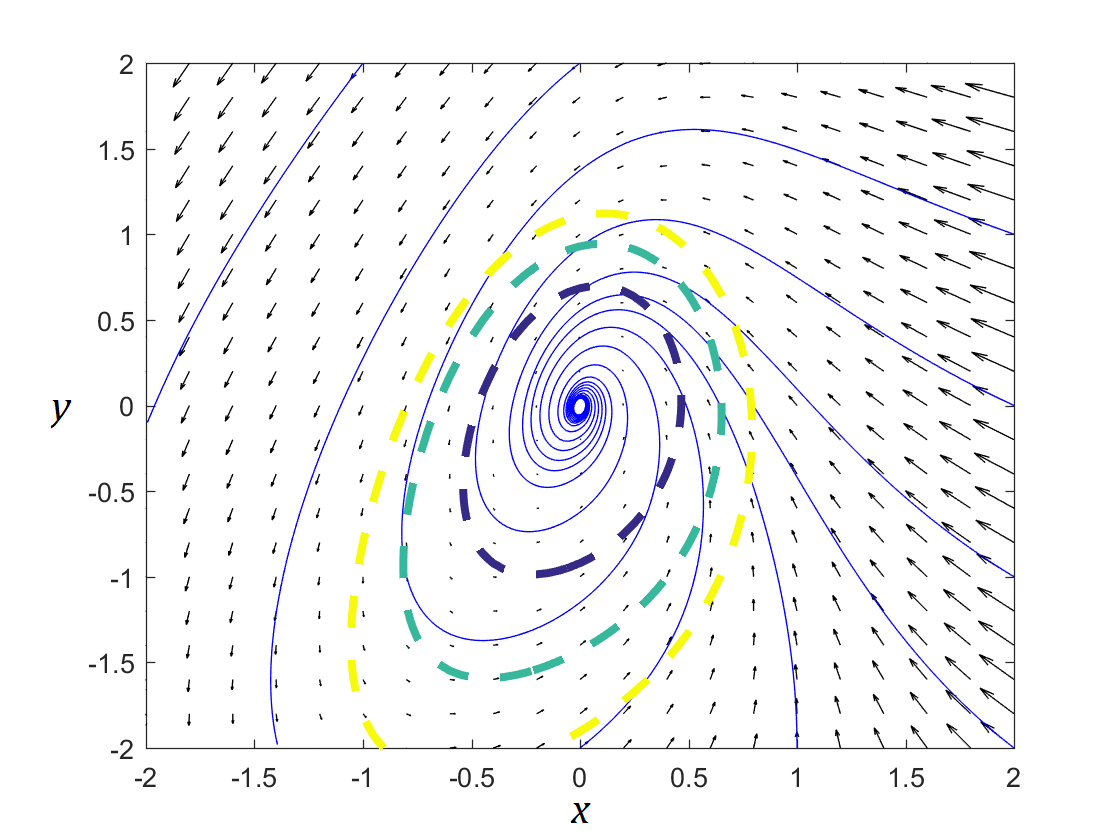}
	\caption{Plot of the vector field given in (\ref{eq:jet}) together with some trajectories (thin lines) and some level sets of an sos Lyapunov function of degree 4 (dashed thick lines) }
	\label{fig:example.stab}
\end{figure}

\subsubsection{Control.} So far, we have seen systems of the type $\dot{x}=f(x)$, i.e., autonomous dynamical systems. As mentioned briefly in the introduction, it can be the case that the dynamics depend on the state $x(t)$ but also on an external output $u(x(t))$, called a \emph{control}, i.e.
$$\dot{x}=f(x(t),u(x(t))).$$ 
One can study many properties of such systems, but if one wants to focus on stability, a natural question to answer is how can one go about designing the controller $u$ in such a way that the size of the region of attraction (i.e., the set of initial states from which a trajectory can start and be asymptotically drawn to its equilibrium) is maximized? We briefly present the results given in  \cite{ControlAppsSOS,ICRA_Acrobot} in this paragraph. We consider a polynomial control affine system
$$\dot{x}=f(x)+g(x)u(x),$$ where $x$ is the state variable, $u(x)$ is the control, and $f,g$ are fixed polynomials that are given to us. If we can find a Lyapunov function $V(x)$ and a sublevel set $B_{\rho}\mathrel{\mathop{:}}=\{x \in \mathbb{R}^n~|~V(x)\leq \rho\}$ of $V$ such that:
\begin{align}\label{eq:roa.control}
x \in B_{\rho}, x\neq 0 \Rightarrow V(x)>0 \text{ and } \dot{V}(x)<0,
\end{align}
then $B_{\rho}$ is a subset of the true region of attraction. To do this, we solve
\begin{align*}
\max_{\rho,L(x),u(x),V(x)} &\rho\\
\text{s.t. } &V(x) \text{ sos}\\
&-\langle \nabla V(x),f(x)+g(x)u(x)\rangle+L(x)(V(x)-\rho) \text{ sos}\\
&L(x) \text{ sos}\\
&V(\sum_{j}e_j)=1,
\end{align*}
where $e_j$ is the $j^{th}$ standard basis vector for the state space $\mathbb{R}^n$, and $\dot{V}(x)=\frac{\partial V(x)}{\partial x}^T (f(x)+g(x)u(x)).$
To see this, note that (\ref{eq:roa.control}) is implied by the first, second, and third constraint. The last constraint is simply a normalization constraint which prevents $\rho$ from getting arbitrarily big by scaling of the coefficients of $V$. Solving this problem is not quite an sos program: indeed, the feasible set is not even convex as we multiply decision variables together (e.g., $L(x)V(x)$). However by alternating optimization over $V$ and $\rho,$ with $u$ and $L(x)$ fixed, and optimization over $\rho,u$ and $L$ with $V$ fixed (using bisection on $\rho$), we are able to solve this problem using sum of squares optimization. Examples of successful implementations of such techniques can be found in the two papers \cite{ControlAppsSOS,ICRA_Acrobot} mentioned above.
\vspace{2mm}

\subsection{Collision avoidance} \label{sec:col.avoid} \index{Collision avoidance}

When Lyapunov theory was first developed, its goal was primarily to certify stability of systems. Thus, Lyapunov functions originally referred to those functions whose properties certified stability of equilibrium points (such as the ones defined in Theorem \ref{th:lyap}). Now, however, the notion of a Lyapunov function has come to englobe any function that is able to certify properties of a system without requiring explicit computation of its trajectories. Following this broader definition, we will present another category of Lyapunov functions in this subsection, sometimes called \emph{barrier certificates}, which prove that systems are \emph{collision-avoidant}.

We consider again a polynomial dynamical system as in (\ref{eq:ode.stability}) and we let $\mathcal{X}_{0}$ and $\mathcal{X}_{u}$ to be two sets in $\mathbb{R}^n$. We assume that the trajectories of our system start in $\mathcal{X}_0$, i.e., $x(0) \in \mathcal{X}_0$. We would like to guarantee that all trajectories of (\ref{eq:ode.stability}) whose initial states $x(0)$ are in $\mathcal{X}_0$ do not enter the ``unsafe region'' $\mathcal{X}_{u}$. Such a system is called collision-avoidant. A sufficient condition for the system to be collision-avoidant is the existence of a barrier certificate, as we describe below.

\begin{theorem}\cite{barrier_certificates_PrajnaJadbabaie} \label{th:barrier}
	Suppose there exists a barrier certificate, namely a continuously differentiable function $B:\mathbb{R}^n \rightarrow \mathbb{R}$ that satisfies the following conditions:
	\begin{enumerate}[(i)]
		\item $B(x)>0$ for all $x \in \mathcal{X}_u$
		\item $B(x)\leq 0$ for all $x \in \mathcal{X}_0$
		\item $\dot{B}(x)=\nabla B(x)^T f(x)\leq 0$ for all $x \in \mathbb{R}^n$.
	\end{enumerate}
	Then there exists no trajectory of (\ref{eq:ode.stability}) that starts from an initial state in $\mathcal{X}_0$ and reaches a state in $\mathcal{X}_u$.
\end{theorem}

\begin{proof}
	Assume that a barrier certificate $B$ satisfying the conditions above exists. Let $x(t)$ be a trajectory in $\mathbb{R}^n$ starting at a point $x(0)$ in $\mathcal{X}_0$ and consider the evolution of $B(x(t))$ along this trajectory. By (ii), $B(x(0)) \leq 0$. Furthermore, the derivative of $B$ along the trajectory is nonpositive from (iii). This implies that $B(x(t))$ decreases with $t$ and hence $B(x(t))$ can never become positive. As any $x \in \mathcal{X}_u$ satisfies $B(x)>0$, it follows that any such trajectory can never reach $\mathcal{X}_u$.
\end{proof}

Just as was done previously, we can search for a barrier certificate within the set of polynomial functions. Under the assumption that the sets $\mathcal{X}_u$ and $\mathcal{X}_0$ are closed basic semialgebraic sets, i.e., can be written as the intersection of a finite number of polynomial equalities or inequalities, we can rewrite constraints (i)-(iii) in Theorem \ref{th:barrier} using sum of squares polynomials. 

\begin{definition}\label{def:sos.barrier}
	Let $\mathcal{X}_u=\{x \in \mathbb{R}^n~|~ g_1(x)\geq 0,\ldots, g_m(x)\geq 0\}$ and $\mathcal{X}_0=\{x \in \mathbb{R}^n~|~ \tilde{g}_1(x)\geq 0,\ldots, \tilde{g}_p(x)\geq 0\},$ where $g_1,\ldots,g_m,\tilde{g}_1,\ldots,\tilde{g}_p$ are polynomials. A sum of squares (sos) barrier certificate is a multivariate polynomial $B$ such that
	\begin{enumerate}[(i')]
		\item $B(x)=\epsilon+\sigma_0(x)+\sum_{i=1}^m \sigma_i(x) g_i(x)$, where $\epsilon >0$ fixed and $\sigma_i, i=0,\ldots,m$ are sum of squares polynomials
		\item $-B(x)=\tau_0(x)+\sum_{i=1}^p \tau_i(x) \tilde{g}_i(x)$, where $\tau_i, i=0,\ldots,p$ are sum of squares polynomials
		\item $-\dot{B}$ is sos.
	\end{enumerate}
\end{definition}
Note that searching for such a polynomial is a semidefinite program and that if such a polynomial exists, then it follows that it is a barrier certificate, and hence that the system is collision avoidant. 

\begin{example}
	We illustrate the ideas in this paragraph via an example given in \cite{barrier_certificates_PrajnaJadbabaie}. Consider the two dimensional polynomial dynamical system 
	\begin{equation}\label{eq:col.avoid.example}
	\begin{aligned}
	\dot{x}&=y\\
	\dot{y}&=-x+\frac{1}{3}x^3-y
	\end{aligned}
	\end{equation}
	and the sets $\mathcal{X}_0=\{(x,y)~|~ (x-1.5)^2+y^2 \leq 0.25\}$ and $\mathcal{X}_u=\{(x,y)~|~ (x+1)^2+(y+1)^2 \leq 0.16\}.$ We wish to show that this system is collision avoidant. With this goal in mind, we search for an sos barrier certificate as defined in Definition~\ref{def:sos.barrier} using YALMIP \cite{yalmip} and find one of degree 4. In Figure \ref{fig:col.avoid}, we have plotted the boundaries of the two sets $\mathcal{X}_0$ and $\mathcal{X}_u$ as well as some trajectories initialized within $\mathcal{X}_0$ and the 0-level set of our barrier certificate. Note that the 0-level set of our barrier certificate can in fact be viewed as a physical barrier.
\end{example}

\begin{figure}[h!]
	\centering
	\includegraphics[scale=0.25]{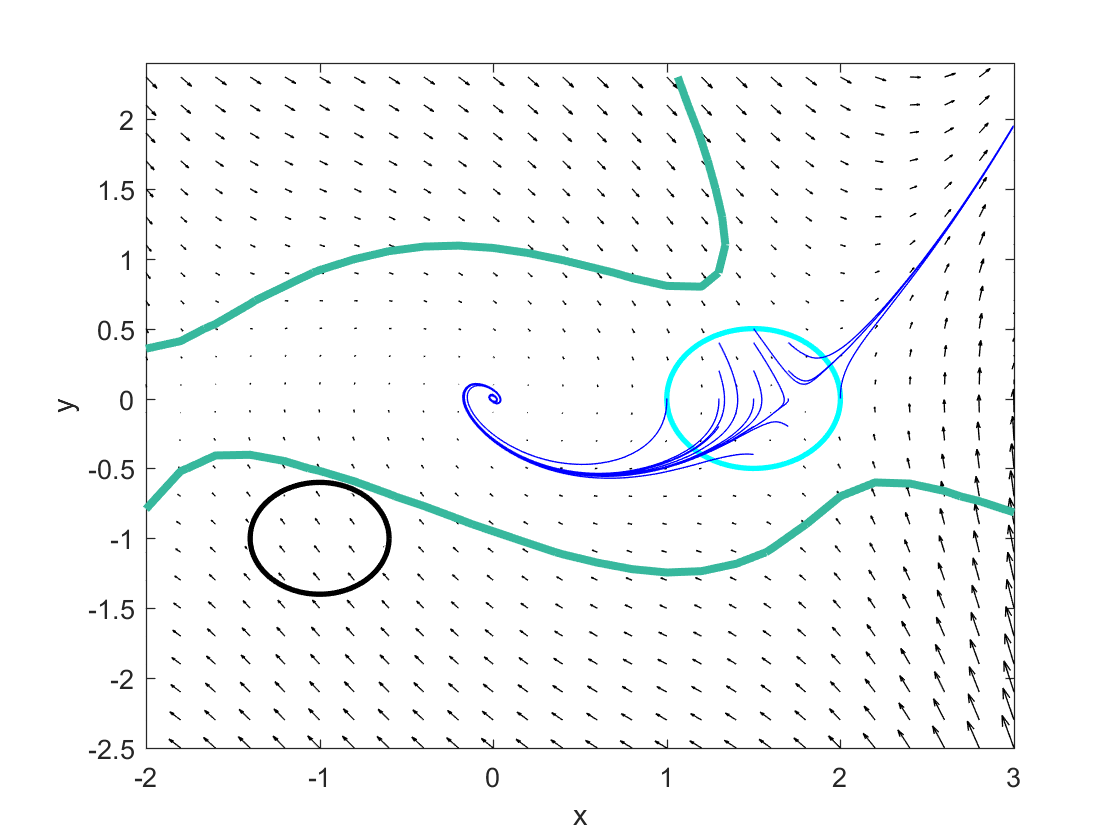}
	\caption{The vector field corresponding to the dynamical system given in (\ref{eq:col.avoid.example}). The initial set is the interior of the light blue circle whereas the unsafe set is the interior of the black circle. Some trajectories are plotted in dark blue. The thick green line represents the 0-level set of $B$. Note that we are guaranteed that no trajectory initialized in the light blue circle will cross the green line, and hence go to the unsafe set.}
	\label{fig:col.avoid}
\end{figure}

\section{Stability of switched linear systems}\label{sec:switched} \index{Switched linear systems} \index{Stability of switched linear systems}

We now transition from a continuous dynamical system to a discrete dynamical system but with a new twist: in this section, we consider discrete linear systems which are both uncertain and time-varying. More specifically, let $$\Sigma\mathrel{\mathop{:}}=\{A_1,\ldots,A_m\}$$ be a set of $m$ real $n\times n$ and let the convex hull of $\Sigma$ be denoted by $$conv(\Sigma)\mathrel{\mathop{:}}=\left\{\sum_{i=1}^m \lambda_i A_i~|~ \lambda_i\geq 0, i=1,\ldots,m, \sum_{i=1}^m \lambda_i=1\right\}.$$ Further define the following discrete-time dynamical system
\begin{align}\label{eq:conv}
x(k+1)=M_k x(k), \text{ where } k=0,1,2\ldots \text{ is the time index and } M_k \in conv(\Sigma).
\end{align}
Note that this dynamical system is linear, but time-varying as the matrix $M_k$ changes with time, and uncertain as, at every time step, we only know that $M_k$ belongs to the convex hull of a set of fixed matrices, without knowing precisely which one it is. We are interested in knowing whether the equilibrium point $\bar{x}=0$ is \emph{absolutely asymptotically stable} (AAS) for (\ref{eq:conv}), i.e., whether $\lim_{k \rightarrow \infty} x(k)=0$ for any $x(0) \in \mathbb{R}^n$ and any sequence of matrices $\{M_k \in conv(\Sigma)\}_k.$ As an example of where such a problem and system may arise, consider, e.g., the task of checking whether a delivery drone is flying in a stable fashion in a windy environment. By linearizing its dynamics around a desired equilibrium point, the behavior of the drone can be modeled locally by a linear dynamical system. However, as this linear dynamical system is unknown due to parameter uncertainty and modeling error, and time-varying due to the effect of the wind, the drone’s behavior is better modeled by a system of the type given in (\ref{eq:conv}).

Define now, for the same family of $m$ matrices $\Sigma$, the following dynamical system, called a \emph{switched linear system}:
\begin{align}\label{eq:switched}
x(k+1)=A_{\sigma(k)}x(k),
\end{align} 
where $k=0,1,2,\ldots$ is the time index and $\sigma:\mathbb{N}\rightarrow \{1,\ldots,m\}$. It so happens that the origin is AAS for (\ref{eq:conv}) if and only if it is \emph{asymptotically stable under arbitrary switching} (ASUAS) for (\ref{eq:switched}). This means that $\lim_{k \rightarrow \infty} x(k)=0$ for any $x(0) \in \mathbb{R}^n$ and any sequence of matrices $A_{\sigma(1)}, A_{\sigma(2)}, \ldots$ In the following, we will study ASUAS for (\ref{eq:switched}) but all our conclusions will naturally hold for AAS of (\ref{eq:conv}).

First, when $m=1$, the set $\Sigma$ is reduced to one matrix $A_1$ and (\ref{eq:switched}) becomes a discrete-time linear system. It is a well-known fact (see, e.g., \cite[Section 2.2]{PabloGregRekha_BOOK}) that a discrete-time linear system is asymptotically stable if and only if the spectral radius of $A_1$ is strictly less than one. This can be checked in polynomial-time. When $m \geq 2$, an analogous characterization holds but with a generalization of the notion of spectral radius from one matrix to a family of matrices called the \emph{joint spectral radius}. \index{Joint Spectral Radius}

\begin{definition}\cite{RoSt60}
	Let $\Sigma=\{A_1,\ldots,A_m\}$ be a family of $m$ matrices of size $n \times n$. The joint spectral radius (JSR) of $\Sigma$ is given by
	\begin{align}\label{def:JSR}
	\rho(\Sigma)=\lim_{k \rightarrow \infty} \max_{\sigma \in \{1,\ldots,m\}^k } ||A_{\sigma(1)}\ldots A_{\sigma(k)}||^{1/k},
	\end{align}
	where $||.||$ is any matrix norm.
\end{definition} 
Note that when $m=1$, this definition collapses into $$\rho(A_1)=\lim_{k \rightarrow \infty} ||A_1^k||^{1/k}.$$ The right hand side is the spectral radius of $A_1$ from Gelfand's formula, hence the joint spectral radius is equal to the spectral radius when $m=1$. As previously mentioned, ASUAS can be characterized using the JSR, which is what we make explicit now.
\begin{theorem}\cite{Raphael_Book}
	The origin is ASUAS for the system given in (\ref{eq:switched}) if and only if $\rho(\Sigma)<1$.
\end{theorem}
Unlike the linear-system case, where one can decide whether the spectral radius of a matrix is less than one in polynomial time, it is not known whether the problem of testing if $\rho(\Sigma)<1$ is even decidable.  The related question of testing whether $\rho(\Sigma)\leq 1$ is known to be undecidable, already when $A$ contains only 2 matrices \cite{BlTi2}. We refer the reader to \cite{BlTi1} for more computational complexity results relating to the JSR. With the previous result in mind, it comes as no surprise that stability of a switched linear system is not implied by all individual matrices in $\Sigma$ having spectral radius less than one. Consider, e.g., $\Sigma=\{A_1,A_2\}$ with
\begin{align*}
A_1=\begin{bmatrix}
0 & 2 \\ 0 & 0
\end{bmatrix}
\text{ and }
A_2=\begin{bmatrix}
0 & 0\\
2 & 0
\end{bmatrix}.
\end{align*}
Observe that the spectral radii of $A_1$ and $A_2$ are zero, which is less than one. However
$$A_1A_2=\begin{bmatrix} 4 & 0\\ 0 & 0 \end{bmatrix}$$
and so $\rho(\Sigma)$ is lower bounded by $\sqrt{4}=2 > 1$, and the switched linear system is not stable.

As a consequence, it is of interest to compute upper bounds on the JSR: if these bounds are strictly less than 1, then it will follow that the JSR is as well and the system will be asymptotically stable. A first theorem in this direction, which provides a stepping-stone towards the use of sum of squares polynomials, is given below.

\begin{theorem}\cite[Theorem 2.2]{Pablo_Jadbabaie_JSR_journal}\label{th:pos.poly.jsr}
	If there exists a positive definite (homogeneous) polynomial $p(x)$ of degree $2d$ that satisfies
	$$p(A_ix) \leq \gamma^{2d} p(x), \forall x\in \mathbb{R}^n, \forall i=1,\ldots,m.$$
	Then, $\rho(A_1,\ldots,A_m) \leq \gamma$.
\end{theorem}

\begin{proof}
	If $p(x)$ is strictly positive, then by compactness of the unit ball in $\mathbb{R}^n$ and continuity of $p$, there exists constants $0<\alpha \leq \beta$ such that $$\alpha ||x||^{2d} \leq p(x) \leq \beta ||x||^{2d} \text{ for all } x \in \mathbb{R}^n.$$
	It follows that
	\begin{align*}
	||A_{\sigma(k)}\ldots A_{\sigma(1)}|| &\leq \max_x \frac{||A_{\sigma(k)}\ldots A_{\sigma(1)}x||}{||x||}\\
	&\leq \left(\frac{\beta}{\alpha} \right)^{1/2d} \max_x \frac{p(A_{\sigma(k)}\ldots A_{\sigma(1)}x)^{1/2d}}{p(x)^{1/2d}}\\
	& \leq \left(\frac{\beta}{\alpha} \right)^{1/2d} \gamma^k.
	\end{align*}
	From the definition of the joint spectral radius given in (\ref{def:JSR}), by taking $k^{th}$ roots and the limit $k\rightarrow \infty$, we immediately have the upper bound $\rho(A_1,\ldots,A_m) \leq \gamma$.
\end{proof}

This theorem clues us in on how to use sum of squares polynomials to compute upper bounds on the JSR. We define, as is done in \cite{Pablo_Jadbabaie_JSR_journal}, the following quantity:
\begin{align}\label{eq:rho.sos}
\rho_{SOS,2d} \mathrel{\mathop{:}}= \begin{bmatrix} \inf_{p \text{ of degree }2d,\gamma} \gamma\\
\text{s.t. } p \text{ sos}\\
\gamma^{2d}p(x)-p(A_ix) \text{ sos}, i=1,\ldots,m,\\
\int_{S^{n-1}} p(x)dx=1,
\end{bmatrix}
\end{align}
where $S^{n-1}$ is the hypersphere.
Note that for fixed $d$ and fixed $\gamma$, the computation of $\rho_{SOS,2d}$ is a semidefinite program. In particular, constraining the integral of $p(x)$ over the hypersphere to be equal to 1 is a linear equation in the coefficients of $p$; see, e.g., \cite{folland2001integrate}. This latter constraint is added on to bypass the aforementioned issue of $p$ and $\gamma^{2d}p(x)-p(A_ix)$ being nonnegative, but not positive. It is easy to see that one can appropriately scale $p$ to satisfy the desired condition without changing the optimal value of the problem. To obtain the smallest $\gamma$ such that $p$ sos and $\gamma^{2d}p(x)-p(A_ix)$ sos, we proceed by bisection on $\gamma$. Indeed, one cannot optimize outright over $\gamma$ and $p$ as the decision variables multiply in the second constraint, making it a nonconvex optimization problem. As a consequence, we typically fix $d$ and then solve a sequence of semidefinite programs as we bisect over $\gamma$. If the optimal value of $\gamma$ found for that $d$ is satisfactory for our purposes, we stop there; otherwise, we move on to a higher degree.

The quality of the bound on the JSR obtained using the sum of squares relaxation described in (\ref{eq:rho.sos}) can be quantified via the following theorem; interestingly, it is independent of the number $m$ of matrices.

\begin{theorem}\cite[Theorem 3.4]{Pablo_Jadbabaie_JSR_journal}\label{th:pablo.jadbabaie}
	The sos relaxation in (\ref{eq:rho.sos}) satisfies
	$$\binom{n+d-1}{d}^{-1/2d} \rho_{SOS,2d} \leq \rho(A_1,\ldots,A_m) \leq \rho_{SOS,2d}.$$	
\end{theorem}

We finish with an illustrative example of the previously-developed techniques.

\begin{example}\label{ex:switched}
	Consider a modification of Example 5.4. in \cite{JSR_path.complete_journal}. We would like to show that the switched linear system defined by the following two matrices 
	$$A_1=\frac{1}{\alpha} \begin{bmatrix} -1 & -1\\ 4 & 0 \end{bmatrix} \text{ and } A_2=\frac{1}{\alpha} \begin{bmatrix} 3 & 3 \\ -2 & 1 \end{bmatrix},$$
	where $\alpha=3.92$ is stable under arbitrary switching. We are able to show using YALMIP that for $2d=6$ and $\gamma=0.9999$, we recover a feasible polynomial $p$ for the SDP given in (\ref{eq:rho.sos}). (It can be checked that all three polynomials appearing in the sos program are positive.) It follows that $\rho(A_1,A_2)\leq 0.9999<1$ and hence the system is ASUAS. We showcase this in Figure \ref{fig:ASUAS} where we have plotted the 1-level set of $p$ together with three random trajectories of the switched system initalized at the same point. Note that all three trajectories flow towards the origin and remain within the 1-sublevel set of $p$.
\end{example}

\begin{figure}[H]
	\centering
	\includegraphics[scale=0.25]{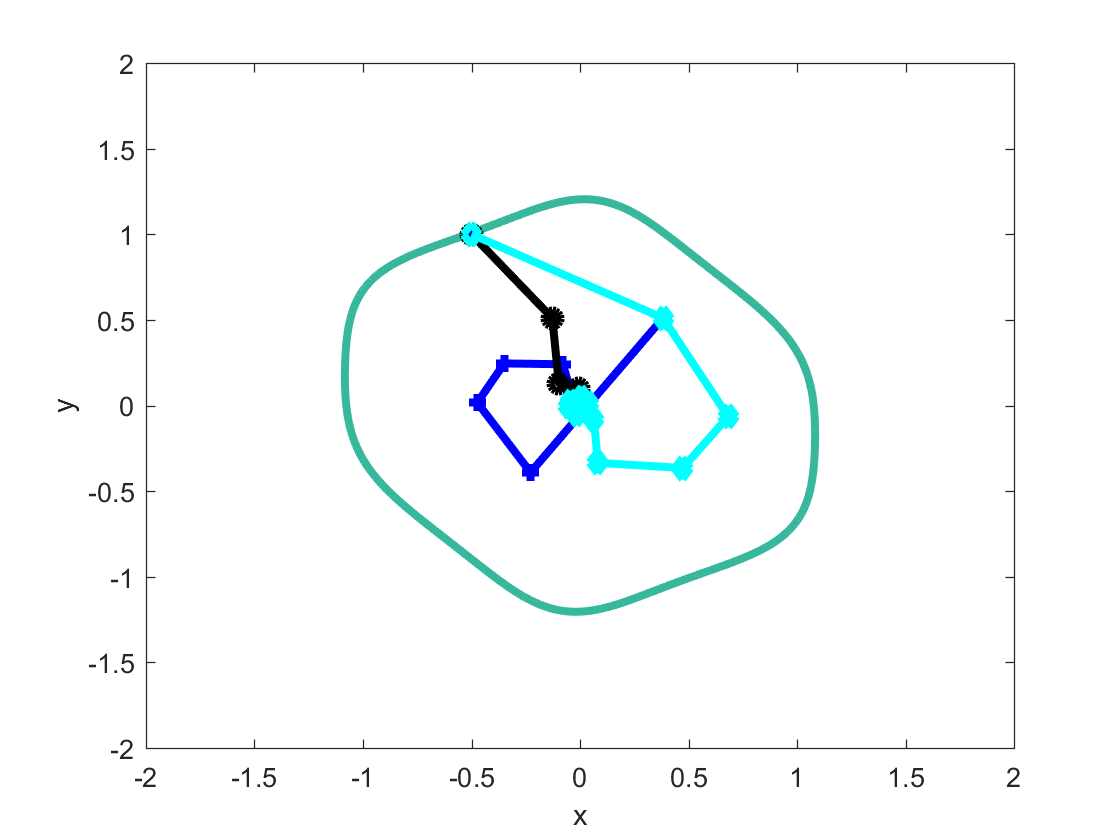}
	\caption{Three possible trajectories of the switched linear system described in Example \ref{ex:switched} together with the 1-sublevel set of the Lyapunov function $p$ obtained (in green)}
	\label{fig:ASUAS}
\end{figure}

\subsection{Using the dual of (\ref{eq:rho.sos}) to generate unstable trajectories}\label{subsec:dual.unstable} \index{Generating unstable trajectories} In the previous subsection, we have seen how one can provide upper bounds on the JSR of a family of matrices in the hopes of certifying stability of an associated switched linear system. In this subsection, our goal is to generate a sequence of matrices whose asymptotic growth rate is arbitrarily close to the JSR. If the system is unstable, then one can hope to produce unstable trajectories via this method, which would serve to certify its instability. The key component to generate these sequences is duality, specifically deriving the dual of (\ref{eq:rho.sos}). We refer the reader to \cite{legat2016generating} for other applications of such techniques as well as a more general version of what is presented here.

Let $\alpha\mathrel{\mathop{:}}=(\alpha_1,\ldots,\alpha_n)^T$ be a vector of nonnegative integers of length $n$ and denote by $|\alpha|\mathrel{\mathop{:}}=\sum_{i=1}^n \alpha_i$. For a vector of variables $x=(x_1,\ldots,x_n)^T$, we can write in shorthand $x^{\alpha}$ to mean $x_1^{\alpha_1}\ldots x_n^{\alpha_n}$. As a consequence, a polynomial $p(x)$ of degree $2d$ in $n$ variables can be written $$p(x)=\sum_{|\alpha|\leq 2d} p_{\alpha}x^{\alpha},$$ where $p_{\alpha}$ is the coefficient of monomial $x^{\alpha}$. There are $N\mathrel{\mathop{:}}=\binom{n+2d}{2d}$ such coefficients and we denote by $\vec{p}$ the $N \times 1$ vector which contains them. For any given $n \times n$ matrix $A$, we have $p(Ax)=\sum_{|\alpha| \leq 2d} q_{A,\alpha}x^{\alpha},$ where $q_{A,\alpha}$ is a linear combination of $p_{\alpha}$. Hence, for any $N \times 1$ vector $\tilde{\mu}$, we can define 
$\tilde{\mu}^{A}$ such that $$\langle \tilde{\mu}^A,\vec{p} \rangle=\langle \tilde{\mu}, \vec{q_A} \rangle.$$ Note that each entry of $\tilde{\mu}^A$ is a linear combination of entries of $\tilde{\mu}$. We use the notation $\Sigma_{n,2d}^*$ in the remainder of the subsection for the dual cone of the set of sum of squares polynomials in $n$ variables and of degree $2d$. We refer the reader to Section \ref{sec:moment.pb} for a definition of it. For our purposes, it suffices to know that this cone is semidefinite representable. The dual problem of (\ref{eq:rho.sos}) can then be written \cite{legat2016generating}
\begin{equation} \label{eq:dual.rho.sos}
\begin{aligned}
&\sup_{\tilde{\mu}_1,\ldots,\tilde{\mu}_m \in \mathbb{R}^N,\lambda \in \mathbb{R} } ~\lambda\\
&\text{s.t.} \sum_{i=1}^m \tilde{\mu}_i^{A_i} -\gamma^{2d}\sum_{i=1}^m \tilde{\mu_i} \in \Sigma_{n,2d}^*\\
&\tilde{\mu}_i \in \Sigma_{n,2d}^*,~\forall i=1,\ldots,m,\\
&\sum_{i=1}^m \langle \tilde{\mu}_i, \vec{s} \rangle=1,
\end{aligned}
\end{equation}
where $\vec{s}$ is the vector of coefficients of $(\sum_{i=1}^n x_i^{2})^d$ in the standard monomial basis. Solving this optimization problem for fixed $\lambda$ and $d$ is a semidefinite program as $\Sigma_{n,2d}^*$ is semidefinite representable. We can proceed as discussed in the previous subsection to obtain the largest $\lambda$ such that the constraints are feasible. 

We now describe the algorithm for recovering a sequence of matrices from $\{A_1,\ldots,A_m\}$ whose asymptotic growth rate is close to the JSR. To initialize the algorithm, find a feasible solution $(\tilde{\mu}_1^*,\ldots,\tilde{\mu}_m^*)$ to (\ref{eq:dual.rho.sos}) and pick $\sigma(0) \in \{1,\ldots,m\}$ such that $\langle \tilde{\mu}_{\sigma(0)}^*,\vec{s} \rangle>0$. Such an index is guaranteed to exist given the last constraint of (\ref{eq:dual.rho.sos}). Then, for $k=1,2,\ldots$, pick an index $\sigma(k) \in \{1,\ldots,m\}$ such that $\langle \tilde{\mu}_{\sigma(k)}^*, \vec{s}_{A_{\sigma(k)} \ldots A_{\sigma(0)} } \rangle$ is maximum, where $\vec{s}_{A_{\sigma(k)} \ldots A_{\sigma(0)}}$ is the vector of coefficients of $s(A_{\sigma(k)} \ldots A_{\sigma(0)}x)$. Repeat the process. The following guarantee on the growth rate of the sequence thus generated can then be shown.

\begin{theorem}[Adapted from Theorem 4.8 in \cite{legat2016generating}]
	Consider a family of $m$ $n \times n$ matrices $\{A_1,\ldots,A_m\}$. For any positive integer $d$ and for any feasible solution $(\tilde{\mu_1}^*,\ldots,\tilde{\mu}_m^*,\lambda^*),$ of (\ref{eq:dual.rho.sos}), the algorithm described above produces a sequence $\sigma(0),\sigma(1),\sigma(2),\ldots$ such that
	$$\lim_{k \rightarrow \infty} ||A_{\sigma(k)}\ldots A_{\sigma(0)}||_2^{1/k} \geq \frac{\lambda^*}{m^{1/2d}}.$$
\end{theorem}

\subsection{Other areas of application of the JSR} \index{Other applications of the JSR} The JSR plays an important role in determining whether a switched linear system is asymptotically stable. But this is far from the only application where it is a relevant quantity. In fact, the concept first started gaining notoriety in the context of the study of wavelets \cite{blondel2008birth}. It also appears in economics \cite{blondel2009polynomial}, coding theory \cite{Raphael_Book}, combinatorics on words \cite{Raphael_Book}, and agent consensus \cite{blondel2005convergence}, to name a few. We give a brief overview of its role in economics and multi-agent consensus here.

\index{Leontief input-output model} In 1973, Wassily Leontief won a Nobel prize in economics for his work on input-output analysis, which models how changes in one sector of the economy can impact other sectors. In his model of inputs and outputs, Leontief divides the economy into $n$ sectors and postulates the following relationship between production and demand:
\begin{align}\label{eq:leontief}
x=Ax+d. 
\end{align}
Here, $d$ is a vector in $\mathbb{R}^n_+$, each component corresponding to demand for the sector $i$, and $x$ is also a vector in $\mathbb{R}^n$, each component describing the production of sector $i$, and $A$ is a nonnegative $n \times n$ matrix, called the consumption matrix, that relates the production of a sector $i$ to the production of other sectors. The interpretation of the consumption matrix is the following: if one wants to produce one unit for sector $i$, then one would need $A_{ij}$ units from sector $j$. The economy is called \emph{productive} if, for any $d\geq 0$, there exists a nonnegative vector $x$ satisfying (\ref{eq:leontief}). This occurs if the spectral radius of $A$ is strictly less than one. However, it can be expected that our knowledge of the consumption matrix is uncertain. It may then be the case that instead of exactly knowing the value of $A$, we simply know that it belongs to the convex hull of matrices $\{A_1,\ldots,A_m\}$. In this case, to determine whether the economy is productive, one needs to consider the joint spectral radius of $\{A_1,\ldots,A_m\}$ instead; see \cite{blondel2009polynomial} for more details.

\index{Multi-agent consensus} The JSR also crops up in the context of multi-agent consensus. Consider a set $N=\{1,\ldots,n\}$ of agents that try to reach agreement on a common scalar value by exchanging tentative values and combining them. More specifically, each agent $i$ starts with a specific value $x_i(0)$ assigned to him or her. The vector $x(t)=(x_1(t),\ldots,x_n(t))$ with the values held by the agents at time $t=0,1,2,\ldots$ is then updated as
$$x(t+1)=A(t)x(t),$$
where $A(t)$ is a stochastic matrix. The goal of \cite{blondel2005convergence} is to establish conditions under which $x_i(t)$ converges to a constant $c$ independent of $i$ when $t\rightarrow \infty$. When this occurs, convergence rates are also shown.  It so happens that a measure of the convergence rate of $x(t)$ to the vector of constants $(c,\ldots,c)$ (when it exists) is given by the joint spectral radius of a set of matrices, obtained by projecting the matrices $A(s), s=0,1,\ldots,t$ onto the space orthogonal to the all ones vector; see \cite{blondel2005convergence} for more details.

\section{Related applications}

\subsection{Fluid dynamics} \index{fluid dynamics}

Fluid dynamics focuses on the study of the flow of fluids such as liquids or gases. Like the problems we considered before, certain properties of this flow can be shown to hold by using Lyapunov analysis. One such property is \emph{stability} of the flow, which we describe now; see \cite{huang2015sum,goulart2012global}. Consider a viscous incompressible flow of velocity $w$ and pressure $p$ evolving inside a bounded domain $\Omega$ with boundary $\partial \Omega$ under the action of body force $f$. The velocity and pressure depend on the point $x$ at which we measure them, as well as time $t$. The changes in velocity and pressure of the flow are described by the Navier-Stokes and continuity equations:
\begin{equation}\label{eq:NS}
\begin{aligned}
\frac{\partial w}{\partial t}+w\cdot \nabla w&=-\nabla p+\frac{1}{R_e} \nabla^2 w+f\\
\nabla \cdot w&=0,
\end{aligned}
\end{equation}
where $\nabla \cdot w$ denotes the divergence of $w$, $\nabla^2 w$ is the vector Laplacian of $w$, and $R_e$ is the Reynolds number, which is a constant that depends on the fluid. We further assume that $w=0$ on the boundary. A steady solution $w=\bar{u}$ and $p=\bar{p}$ to (\ref{eq:NS}) (i.e., a solution to (\ref{eq:NS}) that does not depend on time) is globally stable if for each $\epsilon>0$, there exists $\delta>0$ such that $||w-\bar{u}|| \leq \delta$ at time $t_0$ implies that $||w-\bar{u}|| \leq \epsilon$ for all $t \geq t_0$. In other words, the velocity of the flow does not change too much if a small perturbation is applied, that is, the flow is not turbulent. A steady flow can be proved to be stable if we are able to construct a Lyapunov function $V(u)$ which is positive definite and decreases on any solution $w$ to (\ref{eq:NS}). The goal is then to obtain the largest Reynolds number $R_e$ such that the flow remains stable. This is traditionally done via an energy stability approach, which means that $V(u)$ is taken to be a very specific Lyapunov function, namely $||u||^2$, where $||u||^2=\int_{\Omega} u \cdot u ~d\Omega$. With $R_e$ fixed, determining if the system is stable amounts to solving a linear eigenvalue problem, which is tractable; see \cite{goulart2012global} for more details. However, the value of $R_e$ obtained can be very far from the true value for which the system is unstable. Sum of squares polynomials can then be used to search for improved Lyapunov functions and hence improved values of $R_e$. In \cite{goulart2012global}, it is suggested to use a Lyapunov function of the form $V(a_1,\ldots,a_k,q^2)$ where $u=\sum_{i=1}^k a_i e_i +u_s$, $e_i$ are determined via the energy stability approach, and $q^2=||u_s||^2$. Under these conditions, and by bounding quantities that appear in the Navier-Stokes equations by polynomials in $(a_1,\ldots,a_k,q)$, the authors of \cite{goulart2012global} are able to solve a sum of squares optimization problem to obtain improved lower bounds on the Reynolds number at which turbulence occurs. We refer the reader to \cite{goulart2012global,huang2015sum} and the references within for additional information on this topic.

\subsection{Software verification} \index{Software verification} The goal of software verification is to ensure that a piece of software satisfies some performance and security specifications. This can include for example finite-time termination, avoiding division by zero, or absence of overflow. A way of doing this is to view the computer program as a discrete-time dynamical system: the program operates over a finite state space and is defined by a transition function, which plays the role of a transition map. From there, it is natural to define analogs to Lyapunov functions (so-called \emph{Lyapunov invariants} in \cite{roozbehani2005modeling}), whose purpose it is to certify the aforementioned specifications. We refer the reader to \cite{roozbehani2005modeling} for the formal definitions of these concepts and how to use them in practice.

\part{Probability and measure theory} \index{Probability and measure theory}

In this part, we consider a measure space $(\Omega, \mathcal{F}, \mu)$, where $\Omega \subseteq \mathbb{R}^n$ is the sample space, $\mathcal{F}$ is a $\sigma$-algebra over $\Omega$ and $\mu$ is a measure over $(\Omega, \mathcal{F})$. We remind the reader that a $\sigma$-\emph{algebra} is simply a collection of subsets of $\Omega$ that is closed under complement, as well as countable unions and intersections.  A \emph{measure} is a function $\mu: \mathcal{F} \rightarrow \mathbb{R}^{+} \cup \{+\infty\}$ with two properties: $\mu(\emptyset)=0$ and $\mu(\cup_{i=1}^{+\infty}F_i)=\sum_{i=1}^{+\infty} \mu(F_i)$ for any pairwise disjoint sets $F_i$ in $\mathcal{F}$. If the measure is a \emph{probability} measure, then it has the additional property that $\mu(\Omega)=1$. A recurring concept in this Part will be that of \emph{moments} of a measure. Indeed, as we will see in the next section, nonnegative polynomials and moments of a measure are two sides of the same coin via duality. Let $\alpha\mathrel{\mathop{:}}=(\alpha_1,\ldots,\alpha_n)^T$ be a vector of nonnegative integers of length $n$ and denote by $|\alpha|\mathrel{\mathop{:}}=\sum_{i=1}^n \alpha_i$. For a vector of variables $x=(x_1,\ldots,x_n)^T$, we can write in shorthand $x^{\alpha}$ to mean $x_1^{\alpha_1}\ldots x_n^{\alpha_n}$. The moment of order $\alpha$ of a measure $\mu$ on $(\Omega, \mathcal{F})$ is then given by
$$\int_{\Omega} x^{\alpha} d \mu(x).$$

In Section \ref{sec:bounds}, we will shift our focus from moments of a measure to moments of a \emph{random variable}. This should not confuse the reader as moments of a random variable are in fact moments of a very specific measure---that induced by the random variable. Recall that a random variable $X$ is a (measurable) mapping from $(\Omega,\mathcal{F},\mu)$ to $(E,\mathcal{E})$, where $\mu$ is a probability distribution, $E \subset \mathbb{R}$ and $\mathcal{E}$ is a $\sigma$-algebra over $E$. The probability measure $p_X$ induced by $X$ is then defined as
$$p_X(S)=\mu(\{\omega \in \mathbb{R}^n~|~X(w)\in S\})$$
for any set $S \subseteq E$. Sometimes, the notation $p(X\in S)$ is used as shorthand for $p_X(S)$. The moment of order $k$ of $X$, where $k \in \{0,\ldots,K\}$, of the random variable $X$ is then simply:
$$\int_{\Omega} X^k dp=\int_{\omega \in \Omega} X^{k}(\omega) dp(\omega) \mathrel{\mathop{:}}= \int_{x \in E} x^{k} dp_X(x).$$
Note that here $k$ is \emph{not} a multi-index: this is due to the fact that $p_X$ is a measure over $E$, a subset of $\mathbb{R}$, not over $\Omega$, a subset of $\mathbb{R}^n$.
By definition of the expectation of a random variable, the moment of order $k$ of $X$ can also be viewed as the expectation $E[X^k]$.

In Section \ref{sec:moment.pb}, we will focus on a problem called the \emph{moment problem}, and in Section \ref{sec:bounds}, we will work on computing upper bounds on the quantity $p(X \in B)$ where $B \subseteq E$, and the applications of such methods in finance---more specifically, option pricing.

\section{The moment problem}\label{sec:moment.pb} \index{the moment problem}
In this section, we consider the case where $(\Omega,\mathcal{F})=(\mathbb{R}^n, \mathcal{B})$, with $\mathcal{B}$ being the Borel $\sigma$-algebra over $\mathbb{R}^n$. This is the $\sigma$-algebra generated by the open sets of $\mathbb{R}^n$. The \emph{moment problem} is the following inverse problem: given a sequence $\{y_{\alpha}\}_{\alpha \in \mathbb{N}^n}$ of scalars, does there exist a measure $\mu$ over $(\mathbb{R}^n,\mathcal{B})$ such that
$$y_\alpha=\int_{\Omega} x^{\alpha} d \mu(x), \forall \alpha \in \mathbb{N}^n?$$
When such a measure does exist, we call it a \emph{representing measure} for $y$. For ease of exposition, we will consider a closely related problem: the truncated moment problem. In this case, the sequence $y$ is a truncated sequence, i.e., a sequence $\{y_{\alpha}\}_{\alpha \in \mathbb{N}^n, |\alpha|<c}$ where $c$ is a constant, and we ask again whether there exists a measure $\mu$ such that the (truncated) moments of $\mu$ agree with this sequence. Our presentation mostly follows \cite{Laurent_survey}. For more information on the moment problem, we refer the reader to \cite{Laurent_survey,lasserre2015introduction} and the references therein.

Let $\mathbb{N}^n_{2d}\mathrel{\mathop{:}}=\{\alpha \in \mathbb{N}^n~|~|\alpha| \leq 2d\}$ and define
\begin{align}\label{eq:def.Mnd}
\mathcal{M}_{n,2d} \mathrel{\mathop{:}}=\{\{y_{\alpha}\}_{\alpha \in \mathbb{N}^n_{2d}}~|~ \exists \text{ a measure } \mu \text{ on $\mathbb{R}^n$ such that } y_{\alpha}= \int_{\mathbb{R}^n} x^{\alpha}d\mu, ~\forall \alpha \in \mathbb{N}^n_{2d}\},
\end{align}
i.e., $\mathcal{M}_{n,2d}$ is the set of truncated sequences $\{y_{\alpha}\}$ for which $\{y_{\alpha}\}$ has a representing measure $\mu$. It is easy to see that $\mathcal{M}_{n,2d}$ is a convex cone and that the truncated moment problem is exactly the problem of understanding which sequences belong to $\mathcal{M}_{n,2d}$. To answer this question, we consider the dual cone of $M_{n,2d}$.

\subsection{Dual cone of $M_{n,2d}$}
By definition of a dual cone, we have that
\begin{align}\label{eq:cone.def}
(\mathcal{M}_{n,2d})^*=\{\{p_{\alpha}\}_{\alpha \in \mathbb{N}^{n}_{2d}} ~|~ \sum_{\alpha} p_{\alpha}y_{\alpha} \geq 0, \forall \{y_{\alpha}\} \in \mathcal{M}_{n,2d}\}.
\end{align}
\begin{theorem}\label{th:pos}
	Let $P_{n,2d}$ denote the cone of nonnegative polynomials in $n$ variables and of degree less than or equal to $2d$. We have $P_{n,2d}=(\mathcal{M}_{n,2d})^*.$	
\end{theorem}
\begin{proof}
	Throughout this proof, as mentioned previously, we will identify a polynomial $p$ in $P_{n,2d}$ by its coefficients $p_{\alpha}$ in the standard monomial basis, i.e., $$P_{n,2d} \triangleq \{\{p_{\alpha}\}_{\alpha \in \mathbb{N}^n_{2d}}~|~\ p(x)=\sum_{\alpha} p_{\alpha}x^{\alpha} \geq 0\}.$$
	We first show that $P_{n,2d} \subseteq (\mathcal{M}_{n,2d})^*$. Let $\{p_{\alpha}\}_{\alpha} \in P_{n,2d}$. For any $\{y_{\alpha}\}$ in $(\mathcal{M}_{n,2d})^*$, we have:
	$$\sum_{\alpha} p_{\alpha} y_{\alpha}=\sum_{\alpha} p_{\alpha} \int x^{\alpha} d\mu=\int \sum_{\alpha} p_{\alpha}x^{\alpha} d\mu=\int p(x)d \mu \geq 0$$
	as $p(x)$ is nonnegative and hence the inclusion follows.
	
	We now show that $P_{n,2d} \supseteq (\mathcal{M}_{n,2d})^*$. Suppose $p \notin P_{n,2d}$. Then, there exists $x_0$ such that $p(x_0)<0$. Let $\delta_{x_0}$ be the Dirac measure at point $x_0$ and let $\{y_{\alpha}\}_{\alpha \in \mathbb{N}_{2d}^n}$ be the sequence of moments associated to $\delta_{x_0}$. We have $$\sum_{\alpha} p_{\alpha} y_{\alpha}=\int \sum_{\alpha} p_{\alpha}x^{\alpha} d\delta_{x_0}(x)=\int p(x) d\delta_{x_0}(x)=p(x_0)<0.$$
	Hence $\{p_{\alpha}\} \notin (\mathcal{M}_{n,2d})^*$ and we have shown the converse direction.
\end{proof}
The corollary below follows.

\begin{corollary}\label{cor:dual.pos}
	We have $(P_{n,2d})^*=cl(\mathcal{M}_{n,2d})$, where cl denotes the closure of the set. 	
\end{corollary}

\begin{remark}
We emphasize that we have identified a polynomial $p$ in $P_{n,2d}$ with its coefficients in the standard monomial basis. (Indeed, the dual cone of $\mathcal{M}_{n,2d}$ should be a subset of $\mathbb{R}^{|N^{n}_{2d}|}$.) As a consequence, the interpretation of the dual cone of nonnegative polynomials as sequences of moments is a consequence of having picked a specific basis to represent the polynomials. If that basis changes, e.g., then the interpretation of the dual cone changes as well.
	\end{remark}


From Theorem \ref{th:pos}, we are able to conclude that it is hard to establish when a sequence $\{y_\alpha\}$ belongs to $\mathcal{M}_{n,2d}$ and hence, when a sequence $\{y_{\alpha}\}$ has a representing measure. Indeed, testing membership to the cone of nonnegative polynomials is a hard task, which implies that testing membership to its dual cone is also hard \cite{friedland2016computational}. However, Theorem \ref{th:pos} also gives us a strategy for coming up with necessary conditions for membership to $\mathcal{M}_{n,2d}$. Indeed, let $\Sigma_{n,2d}$ denote the cone of sum of squares polynomials of degree $2d$ in $n$ variables. We have $\Sigma_{n,2d} \subseteq P_{n,2d}$ and so it follows that $(P_{n,2d})^* \subseteq (\Sigma_{n,2d})^*$, and hence:
$$\mathcal{M}_{n,2d} \subseteq (\Sigma_{n,2d})^*.$$
A necessary condition for membership to $\mathcal{M}_{n,2d}$ is then membership to $(\Sigma_{n,2d})^*$. The latter can be tested using semidefinite programming as we see now.

\begin{definition}
	Given a positive integer $d$, and a truncated sequence $y=(y_{\alpha})_{\alpha \in \mathbb{N}^n_{2d}}$, we define the moment matrix $M_d(y)$ to be a symmetric matrix whose rows and columns are indexed by $\alpha \in \mathbb{N}^{n}_{d}$ and where, for $\alpha,\beta \in \mathbb{N}^n_{d}$, the $(\alpha,\beta)^{th}$ entry of the matrix is $y_{\alpha+\beta}$.
\end{definition}

\begin{theorem}
	Let $$\mathcal{M}_{\succeq,n,2d}\mathrel{\mathop{:}}=\{\{y_{\alpha}\}_{\alpha \in \mathbb{N}^{n}_{2d}}~|~ M_d(y) \succeq 0\}.$$
	We have $(\Sigma_{n,2d})^*=\mathcal{M}_{\succeq,n,2d}.$
\end{theorem}

\begin{proof}
	We first show that $\Sigma_{n,2d} \subseteq (\mathcal{M}_{\succeq,n,2d})^*$. By taking the dual, it will follow that $\mathcal{M}_{\succeq,n,2d} \subseteq (\Sigma_{n,2d})^*$. Note that by definition of $(\mathcal{M}_{\succeq, n, 2d})^*$, we have
	\begin{align}
	(\mathcal{M}_{\succeq, n, 2d})^*=\{\{p_{\alpha}\}_{\alpha \in \mathbb{N}_{2d}^n}~|~ \sum_{\alpha} p_{\alpha}y_{\alpha} \geq 0,~ \forall \{y_{\alpha}\} \in \mathcal{M}_{\succeq,n,2d}\}
	\end{align}
	We show that if $p(x)=\sigma^2(x)$, where $\sigma$ is some polynomial of degree less than or equal to $d$ in $n$ variables, then $p \in (\mathcal{M}_{\succeq,n,2d})^*$. It immediately follows that any sum of squares polynomial will belong to $(\mathcal{M}_{\succeq,n,2d})^*$ as it is easy to see that if two polynomials are in $(\mathcal{M}_{\succeq,n,2d})^*$ then their sum is also in $(\mathcal{M}_{\succeq,n,2d})^*$. Once again, we identify $p$ with its coefficients. Let $\vec{\sigma}=(\sigma_{\beta})_{\beta \in \mathbb{N}^n_{d}}$ be the coefficients of $\sigma$. We have $p_{\alpha}=\sum_{\{\beta,\gamma~|~\beta+\gamma=\alpha\}} \sigma_{\beta}\sigma_{\gamma}$ for any $\alpha$, and hence for any $\{y_{\alpha}\} \in \mathcal{M}_{\succeq, n,2d}$,
	$$\sum_{\alpha} p_{\alpha} y_{\alpha}=\sum_{\alpha} y_{\alpha} \sum_{\beta+\gamma=\alpha}\sigma_{\beta} \sigma_{\gamma}=\sum_{\beta}\sum_{\gamma} \sigma_{\beta}\sigma_{\gamma} y_{\beta+\gamma}=\vec{\sigma}^T M_d(y) \vec{\sigma}\geq 0$$
	as $M_d(y) \succeq 0$. So $p \in (\mathcal{M}_{\succeq,n,2d})^*$.
	
	We now show that $(\Sigma_{n,2d})^* \subseteq \mathcal{M}_{\succeq,n,2d}$. Suppose that $\{y_{\alpha}\} \notin M_{\succeq,n,2d}$. This means that $M_d(y)\nsucceq 0$, which implies that there exists a vector $\sigma_0$ such that $\sigma_0^TM_d(y)\sigma_0<0$. Let $\sigma$ be a polynomial with coefficients $\sigma_0$ and let $p=\sigma^2$. Clearly, $p \in \Sigma_{n,2d}$. However, by reprising a similar computation as above, $\sum_{\alpha} p_{\alpha}y_{\alpha}<0$. This means that $\{y_{\alpha}\} \notin (\Sigma_{n,2d})^*$.
\end{proof}

Note that, given a sequence of $|N_{2d}^n|$-tuples of numbers $y_{\alpha}$, one can construct the matrix $M_d(y)$ and check its positive semidefiniteness. If it is not positive semidefinite, then $y_{\alpha}$ does not have a representing measure.

This section can be reworked to take into account measures over arbitrary closed basic semialgebraic sets $K$. The dual of the set of truncated sequences that have a representing measure $\mu$ over $K$ will then simply be the set of polynomials nonnegative over $K$. Furthermore, one can construct stronger necessary conditions for a representing measure to exist over $K$ (under some assumptions) by simply dualizing well-known hierarchies of inner approximations to $P_{n,2d}$ based on sum of squares; see \cite{Laurent_survey} and the references contained there. 

\begin{remark}
	Recently, Barak et al. introduced the concept of \emph{pseudoexpectation}; see, e.g., \cite{barak2014sum}. This can be interpreted in the context of what we have discussed so far. In our results and the proofs of these results, we identified the cone $\Sigma_{n,2d}$ with the set of coefficients of sos polynomials of degree $2d$ and in $n$ variables. Thus the cone $\Sigma_{n,2d}$ we considered was a cone over $\mathbb{R}^{\mathbb{N}^n_{2d}}$. In reality, $\Sigma_{n,2d}$ is a cone over the space of polynomials of degree less than or equal to $2d$, denoted by $\mathbb{R}_{2d}[x]$. The dual cone $(\Sigma_{n,2d})^*$ is then the set of linear functionals $L: \mathbb{R}_{2d}[x] \rightarrow \mathbb{R}$ such that $L(s)\geq 0$ for any $s \in \Sigma_{n,2d}$. Note that there is an isomorphism between this set and $\mathcal{M}_{\succeq,n,2d}$ via the correspondence $L(x^{\alpha})=y_{\alpha}$. 
	
	The pseudoexpectation as defined in \cite{barak2014sum} is simply another name for these linear functionals, with the added constraint\footnote{This latter constraint is because $E[1]=\int d\mu=1$ for a probability measure.} that $L(1)=1$.
	We give the formal definition that appears in \cite{barak2014sum} to contrast: A degree-$l$ pseudoexpectation operator $\tilde{E}$ is a linear operator $L$ that maps polynomials in $\mathbb{R}_l[x]$ into $\mathbb{R}$ and satisfies that $L(1)=1$ and $L(P^2)\geq 0$ for every polynomial $p$ of degree at most $l/2$.
	
	The intuition behind the name is easy to explain. As $\mathcal{M}_{n,2d}$ is the dual (up to closure) of $P_{n,2d}$, it follows that for a measure $\mu$, we should have $$E[p(x)]=\int p(x) d\mu=\sum_{\alpha} p_{\alpha} \int x^{\alpha} d\mu \geq 0,$$ for any nonnegative polynomial $p$. Instead, we have $$\tilde{E}[p(x)] \geq 0$$ for any sum of squares polynomial $p$. Though it resembles its counterpart, $\tilde{E}$ is not actually an expectation: it may be the case that $\tilde{E}[p(x)]<0$ for a nonnegative polynomial $p$, which would not happen if it were truly an expectation. 
\end{remark}

\subsection{The univariate case.} The case where $n=1$ is a noteworthy case for the moment problem (along with the cases $2d=2$ and $(n=2, 2d=4)$) since the set of nonnegative and sum of squares polynomials coincide then. In other words, when $n=1$, $\Sigma_{1,2d}=\mathcal{P}_{1,2d}$. It then follows
from Corollary \ref{cor:dual.pos} that $$(\Sigma_{1,2d})^*=cl(\mathcal{M}_{1,2d}).$$ This gives rise to the following theorem, the formulation of which is taken from \cite{PabloGregRekha_BOOK}.
\begin{theorem}\label{th:moment1}
	Let $y=(y_0,y_1,\ldots,y_{2d})$ be a sequence of real numbers such that $y_0=1$. If $y \in \mathcal{M}_{1,2d}$, i.e., if there exists a probability measure $\mu$ on $\mathbb{R}$ such that $y_i$ is the $i^{th}$ moment of $\mu$, then $y \in (\Sigma_{1,2d})^*$, i.e., $$M_d(y)=\begin{bmatrix}
	y_0 & y_1 & y_2 & \ldots & y_d\\
	y_1 & y_2 & y_3 & \ldots & y_{d+1}\\
	y_2 & y_3 & y_4& \ldots & y_{2d+2}\\
	\vdots & \vdots & \vdots & \ddots & \vdots\\
	y_d & y_{d+1} & y_{d+2} & \ldots & y_{2d}
	\end{bmatrix}$$
	is positive semidefinite. Conversely, if $M_d(y)\succ 0$, then $y$ has a representing probability measure $\mu$, i.e., there exists a probability measure $\mu$ such that $y_i$ is the $i^{th}$ moment of $\mu$.
\end{theorem}

Note that positive definiteness of $M_d(y)$ is needed: one can construct sequences $y$ such that $M_d(y) \succeq 0$ but $y$ does not have a representing measure; see \cite[Remark 3.147]{PabloGregRekha_BOOK}. Furthermore, the theorem above can be extended to measures over intervals of $\mathbb{R}$ rather than measures over the whole of $\mathbb{R}$; for this, see again \cite[Section 3.5.3]{PabloGregRekha_BOOK}. Finally, while this result tells us when a sequence $y$ has a representing measure, it does not directly explain how one should go about constructing such a measure. Some information as to how to do this in practice can be found in \cite[Section 3.5.5]{PabloGregRekha_BOOK}.

\begin{example}
	We check the criterion given in Theorem \ref{th:moment1} on a simple example. Consider the probability measure $\mu$ given by $\mu(dx)=f(x)dx$ where $f(x)$ is the probability distribution function of a standard normal distribution. Let $y=(1,0,1,0,3)$: $y$ is the vector of moments of $\mu$ up to degree $2d=4$. We construct 
	$$M_2(y)=\begin{bmatrix}
	1 & 0 & 1\\
	0 & 1 & 0 \\
	1 & 0 & 3
	\end{bmatrix}.$$
As expected $M_d(y) \succeq 0$.
\end{example}

\subsection{Dual formulation of the polynomial optimization problem}\label{subsec:pop} \index{Dual formulation of the POP} Using the theory developed above, one can view unconstrained polynomial optimization problems (POP) of the type
\begin{align}\label{eq:POP}
\min_{x \in \mathbb{R}^n} p(x),
\end{align}
where $p$ is a polynomial of degree $2d$, or equivalently,
\begin{equation}\label{eq:pop.dual}
\begin{aligned}
&\max_{\lambda} \lambda\\
&\text{s.t. } p(x)-\lambda \in P_{n,2d}.
\end{aligned}
\end{equation}
as the dual problem of an optimization problem over $M_{n,2d}$. Indeed, one can rewrite (\ref{eq:POP}) as
\begin{align*}
\min_{\text{prob measures } \mu \text{ over }\mathbb{R}^n} \int p(x) d\mu.
\end{align*}
To see this, let $p^*=\min_{x \in \mathbb{R}^n} p(x)$. Clearly, as $p \geq p^*$, $\int pd\mu \geq p^*$. Conversely, if $x^*$ is a global minimizer of $p$, then by taking $\mu$ to be the Dirac measure $\delta_{x^*}$ at $x^*$, we get $\int pd\delta_{x^*}=p(x^*)=p^*$.

As $p(x)$ is a polynomial of degree $2d$, we have $p(x)=\sum_{\alpha \in \mathbb{N}_{2d}^n} p_{\alpha} x^{\alpha}$. Plugging into the previous expression, we get 
\begin{align*}
\min_{\text{prob measures } \mu \text{ over }\mathbb{R}^n} \sum_{\alpha \in \mathbb{N}^n_{2d}} p_{\alpha} \int x^{\alpha} d\mu.
\end{align*}
One can then stop dealing with the probability measure $\mu$ itself, but only with the moments $y_{\alpha}\mathrel{\mathop{:}}=\int x^{\alpha} d\mu$, provided that $\{y_{\alpha}\}$ has a representing probability measure. The problem becomes:
\begin{equation}\label{eq:pop.primal}
\begin{aligned}
&\min_{y_{\alpha}} \sum_{\alpha} p_{\alpha} y_{\alpha}\\
&\text{s.t. } \{y_{\alpha}\} \in \mathcal{M}_{n,2d} \text{ and } y_{0}=1.
\end{aligned}
\end{equation}
This problem is dual to (\ref{eq:pop.dual}). Indeed, as $P_{n,2d}=(\mathcal{M}_{n,2d})^*$, we have $p(x)-\lambda \in P_{n,2d} \Leftrightarrow \sum_{\alpha} p_\alpha y_{\alpha} -\lambda y_0 \geq 0,~\forall y_{\alpha} \in \mathcal{M}_{n,2d}$. This latter inequality is then equivalent to $\sum_{\alpha} p_{\alpha}y_{\alpha}\geq \lambda$ as $y_0=1$. We refer the reader to \cite{lasserre_moment} for more information on this dual formulation.

To make the problem tractable, one can replace $\mathcal{M}_{n,2d}$ by its outer approximation $\mathcal{M}_{\succeq, n,2d}$ in (\ref{eq:pop.primal}):
\begin{equation}\label{eq:pop.primal.sdp}
\begin{aligned}
&\min_{y_{\alpha}} \sum_{\alpha} p_{\alpha} y_{\alpha}\\
&\text{s.t. } \{y_{\alpha}\} \in \mathcal{M}_{\succeq,n,2d} \text{ and } y_{0}=1.
\end{aligned}
\end{equation}
 Doing this, we obtain lower bounds on (\ref{eq:pop.primal}) via semidefinite programming. This is equivalent to replacing in the dual (\ref{eq:pop.dual}) $P_{n,2d}$ by $\Sigma_{n,2d}$:
 \begin{equation}\label{eq:pop.dual.sdp}
 \begin{aligned}
&\max_{\lambda} \lambda\\
&\text{s.t. } p(x)-\lambda \in \Sigma_{n,2d},
 \end{aligned}
 \end{equation} 
  which gives us lower bounds on (\ref{eq:pop.dual}) using semidefinite programming. The constrained case where $x \in K$, with $K$ being a closed basic semialgebraic set, can also be considered if we work with measures over $K$ instead, as was mentioned earlier.

\subsubsection{Extracting optimal solutions using the dual}
We present a simple method, described in \cite{parrilo2003minimizing}, to extract optimal solutions of the unconstrained POP (\ref{eq:POP}) under certain conditions. This method does rely on duality, but not sensu stricto on the dual formulation of the POP presented in (\ref{eq:pop.primal}), rather on semidefinite programming duality. As we shall see momentarily, this is because it lifts the initial problem into the space of quadratic functions and each nonnegative quadratic function can be identified with a positive semidefinite matrix. Another method, that applies to constrained POPs and which does rely on the exact dual formulation of the POP presented above (more specifically, on its constrained version) can be found in \cite{henrion2005detecting}. 

We remark that our presentation will be a very high level overview of the material covered in \cite{parrilo2003minimizing}, which we refer the reader to for a more complete and thorough treatment. Consider the unconstrained POP in (\ref{eq:POP}). As $p$ is a polynomial of degree $2d$ and in $n$ variables, it can be written as 
$$p(x)=z(x)^TPz(x)$$
where $z(x)$ is the vector of standard monomials of degree up to $d$ in $n$ variables and $P$ is a symmetric matrix, which is generally not unique. (Throughout we will be working with the standard monomial basis for convenience, but the method remains valid when other bases are used.) As a consequence, (\ref{eq:POP}) is exactly 
$$\min_x z(x)^TPz(x)$$
and if we substitute a vector of variables $z$ for $z(x)$, it is equivalent to
\begin{equation} \label{eq:quad}
\begin{aligned}
&\min_z z^TPz\\
&\text{s.t. } z^TG_0z=1,~z^TG_1z=\ldots=z^TG_rz=0
\end{aligned}
\end{equation}
where $G_0,\ldots,G_r$ are matrices that encode the existing relationships between entries of $z(x)$. As an example, if $p(x)=2+x_1+x_1^4+1.5x_2^4-x_1^3x_2+2x_1^2x_2^2$ and $(z_0,z_1,z_2,z_3,z_4,z_5)$ stands in for $(1,x_1,x_2,x_1^2,x_1x_2,x_2^2)$, (\ref{eq:POP}) can be rewritten as
\begin{align*}
&\min_{z} 2z_0^2+z_0z_1+z_3^2+1.5z_5^2-z_3z_4+2z_3z_5\\
&\text{s.t. } z_0^2=1, z_3z_5=z_2^2, z_1z_2=z_0z_4, \text{ etc}.
\end{align*}
By letting $Z=zz^T$, (\ref{eq:quad}) is equivalent to the following optimization problem:
\begin{equation} \label{eq:sdp.rank}
\begin{aligned}
&\min_Z \mbox{tr}(PZ)\\
&\text{s.t. } \mbox{tr}(G_0Z)=1,~\mbox{tr}(G_1Z)=\ldots=\mbox{tr}(G_rZ)=0\\
&Z \succeq 0, \mbox{rank}(Z)=1.
\end{aligned}
\end{equation}
If we remove the latter condition on the rank, we obtain a relaxation of (\ref{eq:sdp.rank}), which is a semidefinite program. If the optimal solution $Z^*$ happens to be of rank 1, i.e. $Z^*=z^{*}z^{*T}$, then we can extract an optimal solution to the initial problem by reading off the entries of $z^*$ which correspond to $x_1,\ldots,x_n$ in $z(x)$. 

Where does duality come into play here? One can in fact view (\ref{eq:sdp.rank}) (resp. its relaxation) as an analog to (\ref{eq:pop.primal}) (resp. (\ref{eq:pop.primal.sdp})), and so as a ``dual formulation'' of (\ref{eq:pop.dual}). Furthermore, the dual of the relaxation of (\ref{eq:sdp.rank}) is given by
\begin{align*}
&\max_{\lambda, \mu_1,\ldots,\mu_r} \lambda\\
&\text{s.t. } P-\lambda G_0-\sum_{i=1}^r \mu_i G_i \succeq 0
\end{align*} 
which provides lower bounds on (\ref{eq:pop.dual}) and can be viewed as an analog to (\ref{eq:pop.dual.sdp}).

\section{Probability bounds given moments of a random variable and applications to option pricing} \label{sec:bounds}\index{Probability bounds given moments of a random variable}

In this section, we consider a measure space $(\Omega,\mathcal{F},p)$ where $p$ is a probability measure, a random variable $X$ over $(E,\mathcal{E})$ where $E$ is a subset of $\mathbb{R}$ and $\mathcal{E}$ is a $\sigma$-algebra over $E$, and its induced measure $p_X$. Note that though we are considering a random \emph{variable} here (and not a random \emph{vector}), our results can easily be extended to the multivariate case.
Let  $\{y_{k}\}_{k \in \{0,\ldots,K\}}$ be a sequence of scalars and let $\Phi$ be the set of probability measures 
\begin{align}\label{eq:def.Phi}
\Phi\mathrel{\mathop{:}}=\{p_X~|~\int_{E} x^{k} dp_X(x)=y_{k}, \forall k \in \{0,\ldots,K\}\}.
\end{align}
The set $\Phi$ can also be identified with the set of random variables $X$ whose order-$k$ moments coincide with $y_{k}$ for $k \in \{0,\ldots, K\}$. We use the notation $\Phi$ for both sets and assume that $\Phi$ is non-empty (or in other words, $\{y_{k}\}$ always has at least one representing measure). The problem we are considering is then: given a sequence $\{y_{k}\}_{k \in \{0,\ldots, K\}}$ as described above, and a set $S \subseteq E$ described by polynomial inequalities, derive a ``tight'' upper bound on $$p_{X}(S)=\mathrel{\mathop{:}}p(X \in S),$$ i.e., derive $\sup_{p_{X} \in \Phi} p_{X}(S)=\sup_{X \in \Phi} p(X \in S)$.

The problem of using moments of a random variable to upper bound the probability that it belongs to a certain set is a problem that has a rich history within the field of probability theory. Two of the most ubiquitous inequalities in the field, namely that of Markov and Chebychev, are examples of this. Indeed, the Markov inequality states that, for any nonnegative random variable $X$ and positive scalar $a$,
\begin{align}\label{eq:Markov}
p(X \geq a) \leq \frac{E[X]}{a}.
\end{align}
Similarly, the Chebychev inequality states that for any random variable $X$,
\begin{align}\label{eq:Chebychev}
p(|X-E[X]|>t) \leq \frac{var(X)}{t^2}.
\end{align}
Note that the upper bounds on the probabilities in both inequalities only depends on the distribution of the random variable $X$ via its moments.

How can we tackle the more general problem stated above? It is straightforward to see that the problem $\sup_{p_{X} \in \Phi} p_X(S)$ can be formulated as
\begin{equation*}
\begin{aligned}
&\sup_{p_X} \int_{E} \textbf{1}_S dp_X\\
&\text{s.t. } \int x^{k} dp_X(x)=y_{k}, \forall k \in \{0,\ldots, K\},
\end{aligned}
\end{equation*}
where $\textbf{1}_S$ refers to the indicator function of $S$, i.e., a function that takes value $1$ on $S$ and $0$ elsewhere. The dual to this problem reads
\begin{equation}\label{eq:dual.bounds}
\begin{aligned}
&\inf_{\lambda_{k}} \sum_{k=0}^{K} \lambda_{k}y_{k}\\
&\text{s.t. } \sum_{k=0}^{K} \lambda_{k} x^{k} \geq \textbf{1}_S, \forall x \in E.
\end{aligned}
\end{equation}
Indeed, weak duality is readily shown:
$$\int_{E} \textbf{1}_S dp_X \leq \int_{E} \sum_{k}\lambda_{k} x^{k} dp_X=\sum_{k}\lambda_{k} \int_{E} x^{k}dp_X =\sum_{k}\lambda_{k} y_{k}.$$ Strong duality also holds under certain mild conditions; see, e.g., \cite{bertsimas2005optimal}. If we define $\lambda$ to be the polynomial $\lambda(x)=\sum_{k}\lambda_{k} x^{k}$, we can rewrite (\ref{eq:dual.bounds}) as
\begin{equation*}
\begin{aligned}
\min_{\lambda} &\sum_{k} \lambda_{k} y_{k}\\
\text{s.t. } &\lambda(x)-1 \geq 0, \forall x\in S\\
&\lambda(x) \geq 0, \forall x \in E.
\end{aligned}
\end{equation*} 
If we assume that $E$ and $S$ are basic semialgebraic sets, then this problem can be tacked using sum of squares polynomials. Indeed, one can enforce nonnegativity of the polynomials $\lambda(x)-1$ and $\lambda(x)$ over $E$ and $S$ respectively by using the sum of squares certificates of nonnegativity that have previously been covered. In our case, i.e., the case where $X$ is a random variable, this can be done with no loss \cite{PabloGregRekha_BOOK}. For more complex cases such as the multivariate case (i.e., $X$ is a random vector), we refer the reader to \cite{bertsimas2005optimal,lasserre2002bounds}. 
\begin{example}
	We use these ideas to derive tight upper bounds on $p(X\geq a)$ and $p(|X-E[X]|>t)$. In the process, we will see whether the upper bounds provided by the Markov inequality (\ref{eq:Markov}) and the Chebychev inequality (\ref{eq:Chebychev}) are tight.
	
Let $X$ be a nonnegative random variable whose distribution is unknown but its first moment $E[X]$ is known. We would like to find an upper bound on $p(X\geq a)$ for a given scalar $a>0$. Following our previous notation, we have $K=1$, $S$ is $[a,\infty)$, and $E$, which is where $X$ takes its values, is $[0,\infty)$. The fact that $K=1$ implies that $\lambda(x)$ is an affine polynomial, i.e., $\lambda(x)=\lambda_0+\lambda_1x$. Hence the problem to solve is the following:
	\begin{equation*}
	\begin{aligned}
	\min_{\lambda} ~&\lambda_0 +\lambda_1 E[X]\\
	\text{s.t. } &\lambda(x)-1 \geq 0,~\forall x\geq a\\
	&\lambda(x) \geq 0,~\forall x \geq 0.
	\end{aligned}
	\end{equation*}
	(Note that $y_0=1$ as we are considering a probability measure.) One can rewrite the constraints exactly (see \cite[Section 3.3.1]{PabloGregRekha_BOOK}) as
	\begin{equation} \label{eq:Markov.opt}
	\begin{aligned}
	\min_{\lambda,\sigma,\tau,\sigma',\tau'} ~&\lambda_0 +\lambda_1 E[X]\\
	\text{s.t. } &\lambda(x)-1=\sigma +\tau \cdot (x-a), \sigma \geq 0, \tau \geq 0\\
	&\lambda(x)=\sigma'+\tau'\cdot x,  \sigma' \geq 0, \tau \geq 0,
	\end{aligned}
	\end{equation}
	which is a linear program.
	It is quite easy to see that $$\lambda(x)=\frac{1}{a}x$$ is feasible for (\ref{eq:Markov.opt}). Indeed, $\frac{1}{a} \geq 0$ and $\lambda(x)-1=\frac{1}{a}(x-a)$. The value of the objective is then $\frac{E[X]}{a}$. Hence, $p(X \geq a) \leq E[X]/a$. Is this upperbound tight? It is in the case where $E[X]/a \leq 1$. Indeed, in that case define:
	$$X_0=\begin{cases} &a \text{ with probability } E[X]/a \\
	&0 \text{ with probability } 1-E[X]/a
	\end{cases}.$$
	We have that $X_0 \in \Phi$ as $E[X_0]=E[X]$. Furthermore, $p(X_0 \geq a)=\frac{E[X]}{a}$. When $E[X]/a \geq 1$, then the bound that is tight is simply $1$. This is always an upperbound (take $\lambda_0=1$ and $\lambda_1=0$) and it is tight in this case as 
	$X_0=E[X]$ with probability $1$ belongs to $\Phi$ and achieves the bound $p(X_0 \geq a)=1$. Hence, a tight upper bound on $p(X \geq a)$ using first order information is given by 
	$$\sup_{X \in \Phi} p(X\geq a)=\begin{cases} E[X]/a &\text{ if } E[X]/a \leq 1 \\
	1 &\text{ if } E[X]/a >1
	\end{cases}.$$
	The first case corresponds to the Markov bound.

	Now consider a random variable $X$ whose distribution is unknown but whose first and second order moments, $E[X]$ and $E[X^2]$, are known. Given a scalar $t$, we would like an upper bound on 
	$$p(|X-E[X]| \geq t)=p(\{X \geq t+E[X]\} \cup \{X \leq -t+E[X]\}),$$
	that involves only $E[X]$ and $E[X^2]$. In this case, $K=2$, $S=(-\infty, -t+E[X]] \cup [t+E[X],+\infty)$, $E=\mathbb{R}$, and we have $\lambda(x)=\lambda_0 x+\lambda_1 x+\lambda_2 x^2$. The problem can then be written as
	\begin{equation}\label{eq:chebychev.opt}
	\begin{aligned}
	\min_{\lambda} ~&\lambda_0 +\lambda_1 E[X]+\lambda_2 E[X^2]\\
	\text{s.t. } &\lambda(x)-1 \geq 0,~\forall x \in (-\infty, -t+E[X]] \cup [t+E[X],+\infty)\\
	&\lambda(x) \geq 0, \forall x \in \mathbb{R}.
	\end{aligned}
	\end{equation}
	This is equivalent to (see, e.g., \cite[Theorem 3.72]{PabloGregRekha_BOOK})
	\begin{equation}\label{eq:chebychev.opt1}
	\begin{aligned}
	\min_{\lambda,\sigma,\tau,\sigma',\tau'} ~&\lambda_0 +\lambda_1 E[X]+\lambda_2 E[X^2]\\
	\text{s.t. } &\lambda(x)-1= \sigma(x)+\tau \cdot (x-t-E[X]), \sigma \text{ quadratic and sos, } \tau \geq 0,\\
	&\lambda(x)-1= \sigma'(x)+\tau'\cdot (E[X]-t-x), \sigma' \text{ quadratic and sos, } \tau' \geq 0,\\
	&\lambda \text{ sos},
	\end{aligned}
	\end{equation}
	which is a semidefinite program. However, given the simplicity of the case involved, it is easy to get intuition as to what the correct polynomial $\lambda$ should be from (\ref{eq:chebychev.opt}). We take $$\lambda(x)=\left(\frac{x-E[X]}{t}\right)^2.$$
	It is immediate that $\lambda$ is sos and furthermore, we have $$\lambda(x)-1=\left(\frac{x-t-E[X]}{t}\right)^2+\frac{2}{t}(x-t-E[X])$$ and $$\lambda(x)-1=\left(\frac{E[X]-t-x}{t}\right)^2+\frac{2}{t}(E[X]-t-x)$$
	with $2/t \geq 0$. It follows that $\lambda$ is a feasible solution to (\ref{eq:chebychev.opt}) achieving the bound of $$\frac{var(X)}{t^2}.$$ So, $var(X)/t^2$ is always a valid upper bound on $p(|X-E[X]|>t)$. Is it tight? Again, the answer is yes, but only when $var(X) \leq t^2$. Indeed, consider
	$$X_0=\begin{cases} E[X]+t &\text{ with probability } \frac{var(X)}{2t^2} \\
	E[X]-t &\text{ with probability } \frac{var(X)}{2t^2}\\
	E[X] &\text{ with probability } 1-\frac{var(X)}{t^2}
	\end{cases}.$$
	We have $X_0 \in \Phi$ as $E[X_0]=E[X]$ and $E[X_0^2]=E[X^2]$. Furthermore, $p(|X-E[X]|>t)=var(X)/t^2$. In the case where $var(X) \geq t^2$, a tight upper bound is $1$. It is easy to see that $1$ is always a valid upper bound by taking $\lambda_0=1,\lambda_1=0,\lambda_2=0$ and it is tight when $var(X)\geq t^2$ as we can take $$X_0=\begin{cases} E[X]+\sqrt{var(X)} &\text{ with probability } 1/2 \\ E[X]-\sqrt{var(X)} &\text{ with probability } 1/2
	\end{cases}.$$
	We have $X_0 \in \Phi$ as $E[X_0]=E[X]$ and $E[X_0^2]=E[X^2]$. Furthermore $p(|X_0-E[X]| \geq t)=1$. Putting everything together, a tight upper bound on $p(|X-E[X]|\geq t)$ using first and second order information is given by 
	$$\sup_{X \in \Phi} p(|X-E[X]| \geq t)=\begin{cases} var(X)/t^2 &\text{ if } var(X)/t^2 \leq 1\\ 1 &\text{ if } var(X)/t^2 \geq 1 \end{cases}.$$
	The first case is the Chebychev inequality.
\end{example}

\subsection{Applications to option pricing.} \index{Option pricing} Let $X$ be the (random) price of an asset and $p_X$ its probability distribution. Though $p_X$ is unknown, we assume that the first and second order moments of $X$, which we denote by $y_1$ and $y_2$, are known (e.g., estimated from past data). The zero-th order moment of $X$ is trivially $y_0=1$. We now consider a European call option on the asset with strike price $k$. Recall that a European call option is a derivative security which gives the buyer of the call two options on the day it expires: either (s)he buys a fixed amount of the asset at price $k$, or (s)he does nothing. Hence, the payoff of the buyer of the option will be $\max(0,X-k)$ where $X$ is the price of the asset on the day the call expires: indeed, if the price of the asset is greater than $k$, then the buyer will use his or her option to get it at the reduced price of $k$, thus making $X-k$; if the price of the asset is less than $k$ however, then the buyer will chose to not use his or her option, thus making $0$. A fair price for this option would be
$$E_{p_X}[\max\{0,X-k\}],$$
where the expectation is taken with respect to the unknown probability distribution of $X$. Note that with such a price, the seller does not make a profit on average, but simply breaks even. However, to hedge against uncertainty in the distribution of $X$, the seller choses to pick 
$$\sup_{p_{X} \in \Phi} E_{p(X)}[\max\{0,X-k\}]$$
where $\Phi$ is as in (\ref{eq:def.Phi}) with $E=[0,+\infty)$ (the price of the asset is always nonnegative) and $K=2$. As discussed before, the problem above can be formulated as:
\begin{equation*}
\begin{aligned}
\max_{p_X} &\int_{\mathbb{R}^+} \max\{0,x-k\} dp_X(x)\\
\text{s.t. } &\int_{\mathbb{R}^+} x^k dp_X(x)=y_i, i=0,1,2.
\end{aligned}
\end{equation*}
The dual to this problem is 
\begin{equation*}
\begin{aligned}
\min_{\lambda_k} &\sum_{k=0}^2  \lambda_k y_k\\
\text{s.t. } &\sum_{k=0}^2 \lambda_k x^k \geq \max\{0,x-k\}, \forall x\in \mathbb{R}^+.
\end{aligned}
\end{equation*}
This is equivalent to 
\begin{equation*}
\begin{aligned}
\min_{\lambda_k} &\sum_{k=0}^2  \lambda_k y_k\\
\text{s.t. } &\sum_{k=0}^2 \lambda_k x^k \geq 0, \forall x \in [0,k]\\
&\sum_{k=0}^2 \lambda_k x^k \geq x-k, \forall x \in [k,+\infty),
\end{aligned}
\end{equation*}
which can be solved using semidefinite programming. We refer the interested reader to \cite{bertsimas2002relation} for other examples of problems of this type. Other areas where optimal bounds on probabilities of events can be useful are decision analysis \cite{smith1995generalized} and queuing theory \cite{whitt1984approximations}.

\part{Statistics and machine learning} \index{Statistics and machine learning}

Both of the sections in this part revolve around \emph{regression}. Regression is one of the most fundamental problems in statistics, with applications in many different areas, including the social and physical sciences. We briefly review the problem here. Let $\{(x_i,y_i)\}_{i=1,\ldots,m}$ be a series of data points with $x_i \in \mathbb{R}^n$ being a feature vector, and $y_i \in \mathbb{R}$ being the output variable. We denote by $x_i^j$ the $j^{th}$ component of the vector $x_i$. It is assumed that there is a relationship between $x_i$ and $y_i$ of the form $$y_i=f(x_i)+\epsilon_i,~i=1,\ldots,m$$ where $\epsilon_i$ is some random noise with $E[\epsilon_i]=0$, finite variance, and $\epsilon_i$ independent from $\epsilon_j$. The goal of regression is to find a function $f$ (called a \emph{regressor}) within a class of functions $\mathcal{F}$ such that the error between $f(x_i)$ and $y_i$ is minimized. The notion of error can be, e.g., that of least squares error, which gives us the problem
\begin{equation}\label{eq:regression}
\min_{f \in \mathcal{F}} \sum_{i=1}^m (y_i-f(x_i))^2.
\end{equation}
When $\mathcal{F}$ contains functions that are completely described by a set of parameters $\theta \in \mathbb{R}^p$, the regression is called \emph{parametric} and the optimization can be done over the parameters instead of over $\mathcal{F}$. The case where $$\mathcal{F}=\{f~|~f(z)=\theta_0+\theta_1z_1+\ldots+\theta_n z_n, \text{ where } \theta_0,\ldots,\theta_n \in \mathbb{R}\},$$ for example, is linear regression and finding $f$ amounts to solving an unconstrained convex quadratic program. In Section \ref{sec:shape.constr.reg}, we will show how sum of squares polynomials can be used to enforce shape constraints on the regressor $f$. In Section \ref{sec:opt.design}, we will see how we can use sum of squares techniques to optimally pick the feature vectors $x_i$.

\section{Shape-constrained regression}\label{sec:shape.constr.reg} \index{Shape-constrained regression}

In this section, we consider the regression problem described above and assume that the feature vectors $x_i$ belong to a full-dimensional box $B$ in $\mathbb{R}^n$. Our goal is to fit a function $f$ to the data that minimizes the least squares error within a class of functions $\mathcal{F}$ that have a specific shape (e.g., convex over the box $B$ or monotonous over $B$ in certain directions). We call this problem \emph{shape-constrained} regression. Shape-constrained regression is a very natural problem to consider. In economics for example, if one wants to model a utility function by fitting a regressor to data, then it would make sense to enforce concavity of the regressor. Likewise, we can readily imagine that a number of outputs would depend monotonically on inputs (think, e.g., of the Body Mass Index (BMI) of a person with respect to his or her calorie intake, or the quantity of honey produced in a hive as a function of number of bees). Because of its omnipresence, there have been a number of methods developed to address this problem; see \cite{gupta2016monotonic, hannah2013multivariate,seijo2011nonparametric,lim2012consistency,mazumder2017computational,ghall}. Here, we consider a method that relies on sum of squares programming, developed in, e.g., \cite{magnani2005tractable,mihaela}. One of its main advantages is that it scales polynomially in the number of features of the problem, which is often a caveat in other methods. We discuss it in more depth below. 

Let's consider first the case where we would like to enforce monotonicity of our regressor over $B$ with respect to component $j$, i.e., we want $z_j \mapsto f(z_1,\ldots,z_{j-1}, z_j,z_{j+1},\ldots,z_n)$ to be increasing for all $(z_1,\ldots,z_{j-1},z_{j+1},\ldots,z_n)$ in the appropriate domain. We assume here that $f$ is continuously differentiable. This is then equivalent to imposing that $$\frac{\partial f(z)}{\partial z_j} \geq 0, ~\forall z \in B.$$ If $\rho \in \mathbb{R}^n$ is a vector that encodes the monotonicity profile of $f$ with respect to each one of its variables, i.e., $\rho_j=1$ (resp. $0$, $-1$) if $f$ is increasing (resp. non-monotonic, decreasing) with respect to component $j$, then the monotonicity-constrained regression problem can be written:
\begin{equation*}
\begin{aligned}
&\min_{f} \sum_{i=1}^m (y_i-f(x_i))^2\\
&\text{s.t. } \rho_j\frac{\partial f(z)}{\partial z_j} \geq 0, ~\forall y \in B.
\end{aligned}
\end{equation*}
To make the problem amenable to computation, we restrict ourselves to searching over the space of polynomial functions of degree $d$, i.e., $f$ is assumed to be a polynomial of degreee $d$. The problem remains hard to solve however because of the nonnegativity constraint over the box. Indeed, one can show that even testing whether a polynomial $f$ of degree $d$ has monotonicity profile $\rho$, over a box $B$ is NP-hard, for $d$ as low as 3 \cite{mihaela}. We consequently replace the nonnegativity constraint by a constraint that involves sum of squares polynomials---see \cite[Section 3.4.4]{PabloGregRekha_BOOK} for different ways to do this---and the problem becomes a semidefinite program. The theorem below gives an idea as to the quality of these successive approximations.
\begin{theorem}\cite{mihaela}\label{th:mon}
	Let $f$ be a continuously differentiable function with a given monotonicity profile $\rho$ over $B$. For any $\epsilon>0$, there exists an integer $d$ and a polynomial $p$ of degree $d$ such that $$\max_{x \in B} |f(x)-p(x)|<\epsilon$$ and such that $p$ has same monotonicity profile $\rho$ over $B$. Furthermore, this monotonicity profile can be certified using a sum of squares certificate.
\end{theorem}

Let us consider now the case where we would like to enforce convexity of our regressor $f$ over $B$. We assume that $f$ is twice continuously differentiable and that $H_f$ denotes the Hessian of $f$. This is then equivalent to imposing $$H_{f}(y) \succeq 0, \forall y \in B,$$ 
which is in turn equivalent to $$z^TH_f(y)z \geq 0, ~\forall z \in \mathbb{R}^n,~\forall y\in B.$$
Hence the convexity-constrained regression problem can be written
\begin{equation*}
\begin{aligned}
&\min_{f} \sum_{i=1}^m (y_i-f(x_i))^2\\
&\text{s.t. } z^TH_f(y)z \geq 0,~\forall z \in \mathbb{R}^n,~\forall y \in B.
\end{aligned}
\end{equation*}
We follow the same scheme as previously: we restrict ourselves to polynomial functions, and then replace the nonnegativity constraint of the polynomial (in $z$ and $y$) $z^TH_f(y)z$ by a constraint that involves sum of squares polynomials. Indeed, as before, the problem of testing whether a polynomial of degree $d$ is convex over a box is NP-hard, even for $d=3$  \cite{ahmadi2018complexity}. One can show an analogous result to Theorem \ref{th:mon} for convexity-constrained regression.

\begin{remark}
	Let $f$ be a polynomial in $n$ variables. If $y^TH_f(x)y$ is constrained to be a sum of squares (as a polynomial in $x$ and $y$) then $f$ is said to be \emph{sos-convex}. This is a sufficient condition for (global) convexity as $y^TH_f(x)y$ sos implies that $y^TH_f(x)y \geq 0, \forall x,y \in \mathbb{R}^n$, which implies that $H_f(x) \succeq 0, \forall x \in \mathbb{R}^n$. Optimizing over the set of sos-convex polynomials is a semidefinite program; see \cite{AAA_PP_table_sos-convexity} for more information on the concept of sos-convexity. The above example uses a variant of sos-convexity: we wish to find sufficient conditions for convexity \emph{over a box}. 
	
	Sos-convexity is a notion that can be used in many other applications. For example, some problems in computer vision require fitting a convex shape to a cloud of $3D$ data points in such a way that the volume of this shape is as small as possible: this can be done e.g. by requiring that these points belong to the sublevel set of an sos-convex polynomial; see \cite{magnani2005tractable,ahmadi2016geometry}.
\end{remark}
\begin{remark}
	It goes without saying that both types of constraints (monotonicity and convexity) can be combined if one happens to have the appropriate information.
\end{remark}

\begin{example}
	We now give an example, taken from \cite{mihaela}, relating to the prediction of weekly wages from past data. The data used comes from the 1988 Current Population Survey and is freely
	available under the name ex1029 in the Sleuth2 R package \cite{sleuth}. It contains $25361$ observations and $2$ numerical features: years of experience and years of education. We expect wages to increase with respect to years of education and be concave with
	respect to years of experience. We run both an unconstrained polynomial regression (denoted by UPR), i.e., $\mathcal{F}$ is the set of polynomials of a certain degree in (\ref{eq:regression}), and a convexity-constrained and monotonocity-constrained regression (denoted by Hybrid and described above) on the data. This is done by computing the Root Mean Squared Error (RMSE) for the data with 10-fold cross validation. The results are given in Figure \ref{fig:rmse_wage} with varying degrees of the polynomial regressor. Note that for the training data, obviously UPR performs better than Hybrid as it is less constrained and can overfit. The Hybrid method however has a much better generalization error than UPR. 
\end{example}

\begin{figure}[h!]
	\centering
	\subfigure[RMSE on training data]{\includegraphics[width = 0.49\textwidth]{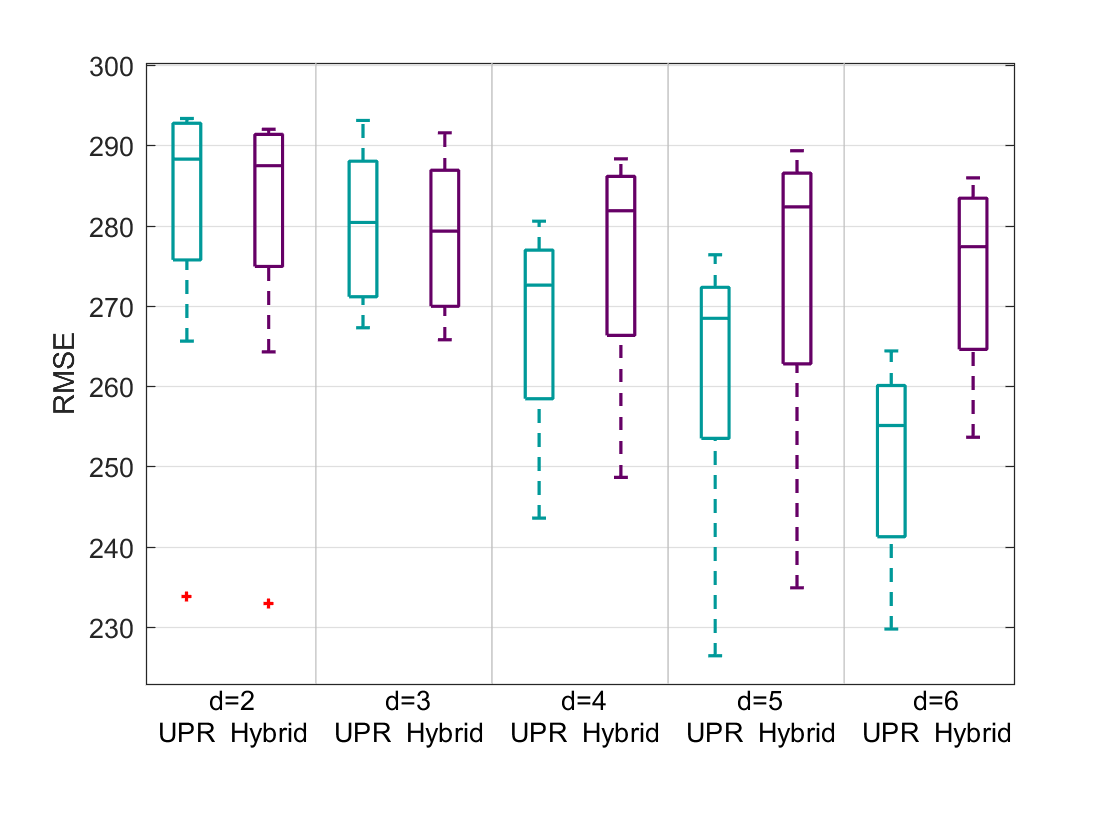}}
	\subfigure[RMSE on testing data]{\includegraphics[width = 0.49\textwidth]{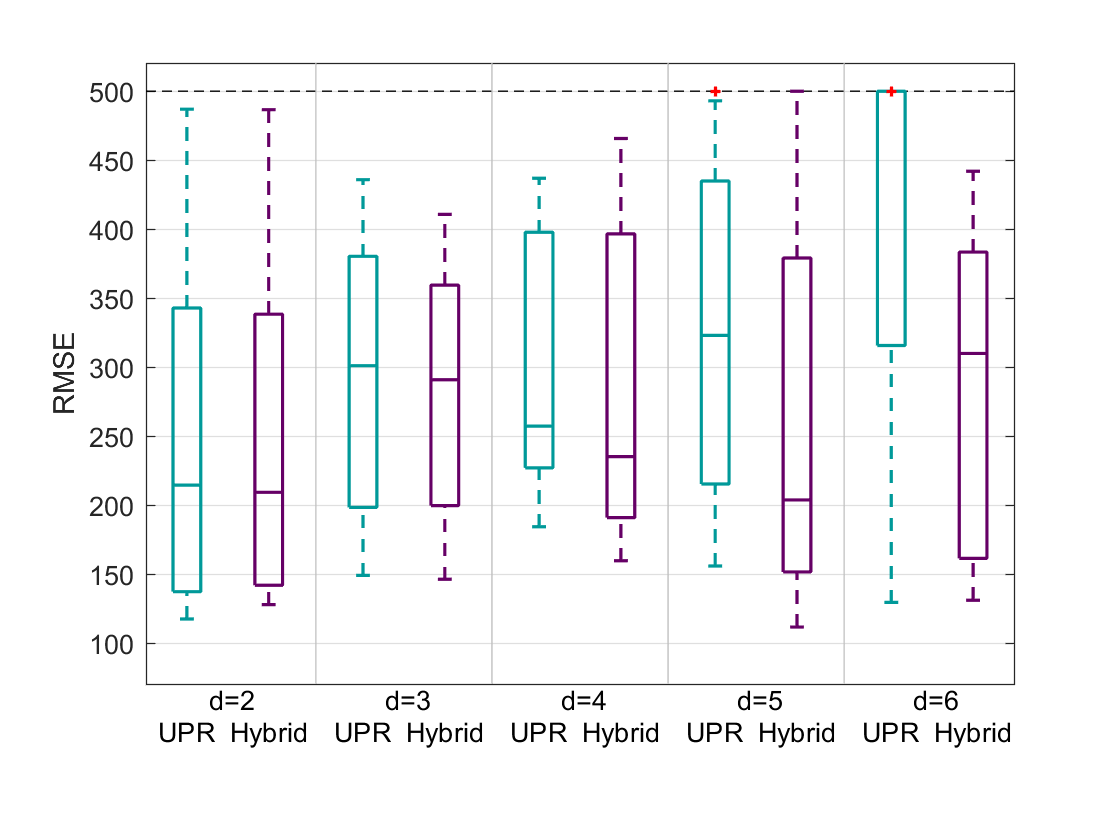}} 
	\caption{Comparative performance of UPR and Hybrid on testing and training sets.}
	\label{fig:rmse_wage}
\end{figure}

\section{Optimal design}\label{sec:opt.design} \index{Optimal design}

We once again consider a regression setting, but this time we are interested in the problem of generating data. Recall that the input to a regression problem are pairs $\{(x_i,y_i)\}_{i=1,\ldots,m}$, where $x_i \in \mathbb{R}^n$ and $y_i \in \mathbb{R}$. Statisticians make a difference between the case where the person conducting the study can choose the feature vectors $\{x_i\}_i$, and the case where the feature vectors $\{x_i\}_i$ are imposed. The latter case is called an \emph{observational study}. An illustrative example is that of studying the impact of the amount of cigarettes smoked on the development of lung cancer: our data will contain the amount of cigarettes that each participant chooses to smoke, without our being able to impact this. Indeed, it would be a major ethical breach if we were asking participants to smoke more, e.g., to change our input data. 

Of interest in this section is the other case, namely the case where the $x_i$ can be fixed to certain values by the experimenter. This is called an \emph{experimental study}. As an example of such a study, consider the problem of measuring the degree of corrosion of steal under the effects of humidity and temperature. By placing a piece of steal in an environment controlled for humidity and temperature, one is able to obtain the degree of corrosion for any values of humidity and temperature that one wishes to have. This set-up is particularly interesting to statisticians as it enables the experimenter to choose advantageous values of the features. The process of choosing such values is termed \emph{experimental design}. In this section, we will be interested in using sum of squares techniques to understand how to design experiments in an optimal way. We follow the presentation given in \cite{de2017approximate}.

We use similar terminology to what is used in Section \ref{sec:moment.pb}: we let $\mathbb{N}_{d}^n=\{\alpha \in \mathbb{N}^n~|~ \alpha_1+\ldots+\alpha_n\leq d\}$ and, for $z=(z_1,\ldots,z_n)^T$, we use the shorthand $z^{\alpha}$ to mean $z_1^{\alpha_1}\ldots z_n^{\alpha_n}$. We consider a parametric regression setting here, in fact, a polynomial one. More specifically,
$$y_i=\sum_{\alpha \in \mathbb{N}^n_{d}} \theta_{\alpha} x_i^{\alpha}+\epsilon_i,~i=1,\ldots,m$$
where $\{\theta_{\alpha}\}$ are the coefficients of the polynomial, and $\epsilon_i$ is random noise with $E[\epsilon_i]=0$, $var(\epsilon_i)=\sigma^2<\infty$, and $\epsilon_i$ independent from $\epsilon_j$. (Recall that $x_i \in \mathbb{R}^n$.) We assume that $m \geq \binom{n+d}{n}$ and that the vectors $\{x_i\}$ can be picked within a compact set $\mathcal{X} \subset \mathbb{R}^n$, described by a finite number of polynomial inequalities. Note that it may be the case that the optimal way of picking the vectors $\{x_i\}$ involves repeating the same value twice, e.g., $x_1=x_2$. As a consequence, our goal is to choose $l$ ($\l \leq m$) distinct values, $t_1,\ldots,t_l \in \mathcal{X}$ that the $\{x_i\}$ take, together with the number of times $n_k$ that value $t_k$ is taken. We let $w_k=\frac{n_k}{m}$, this information can be summarized in a $n+1\times l$ \emph{design matrix}
\begin{align}\label{eq:design}
\xi=\begin{bmatrix}
t_1 & \ldots & t_l\\
w_1 & \ldots & w_l
\end{bmatrix}, \text{ where } w_k=\frac{n_k}{m},
\end{align}
which is what we would like to obtain at the end of the optimization process. In the rest of this section, for convenience, we denote the vector of standard monomials of degree up to $d$ in $n$ variables by $z_d(x)=(1,x_1,x_2,\ldots,x_n,\ldots,x_n^d)$, and by $\theta$ the corresponding vector of coefficients so that $\sum_{\alpha \in \mathbb{N}^n_{2d}} \theta_{\alpha}x_i^\alpha=\theta^Tz_d(x)$. For ease of exposition, we also assume, until otherwise mentioned, that $l=m$ (i.e., the feature vectors $x_i$ are distinct), and that $\sum_{k=1}^m z(t_k)z(t_k)^T \succ 0$.

What should be the objective when picking $\xi$? This depends on what we would like to achieve. In our case, assuming our estimator for $\theta$ is the least squares estimator 
\begin{align}\label{eq:lse}
\hat{\theta}=\arg \min_{\theta} \sum_{k=1}^m (y_k-\theta^Tz(t_k))^2,
\end{align} 
it may be of interest to minimize, in some sense, the variance of $\hat{\theta}$. Indeed, as will be made evident soon, under the assumptions we have on $\epsilon_i$, $\hat{\theta}$ is an \emph{unbiased} estimator of $\theta$, which means that $E[\hat{\theta}]=\theta$; i.e., on average both quantities are equal. It may then be of interest to ask that $\hat{\theta}$ deviate as little as possible from $\theta$ overall: this is exactly equivalent to minimizing the variance of $\hat{\theta}$. Of course, the variance of $\hat{\theta}$ is a matrix here as $\hat{\theta}$ is a vector, so when we claim to minimize the variance of $\hat{\theta}$, we actually mean minimizing the 2-norm of its covariance matrix, or some other measure. 

Consequently, the next step in the process is to obtain an explicit expression of $\hat{\theta}$ and then compute its covariance matrix. As a byproduct, we will obtain unbiasedness of $\hat{\theta}$ as an estimator of $\theta.$ 
Note that under our assumptions, the objective function in (\ref{eq:lse}) is a strictly convex quadratic function and hence has a unique minimum. Using the first-order necessary condition for optimality, we obtain $$\hat{\theta}=\sum_{k=1}^m y_k\left(\sum_{k=1}^m z(t_k)z(t_k)^T\right)^{-1}z(t_k).$$
Replacing $y_k=\theta^Tz(x_k)+\epsilon_k$ in the previous expression, we get
$$\hat{\theta}=\sum_{k=1}^m \left(\sum_{k=1}^m z(t_k)z(t_k)^T\right)^{-1}(\theta^T z(t_k)+\epsilon_k)z(t_k)=\theta+\left(\sum_{k=1}^m z(t_k)z(t_k)^T\right)^{-1} \sum_{k=1}^m \epsilon_k z(t_k).$$
From this, we infer that $E[\hat{\theta}]=\theta$ as $E[\epsilon_k]=0, \forall k=1,\ldots,m$, and hence that $\hat{\theta}$ is an unbiased estimator of $\theta$ as previously claimed. Using standard covariance and variances identities as well as the properties of $\epsilon_i$, it follows that the variance matrix of $\hat{\theta}$ is given by
$$\Sigma=\sigma^2 \left(\sum_{k=1}^m z(t_k)z(t_k)^T\right)^{-1}.$$
We write $\Sigma=\Sigma(\xi)$ as one can see that the covariance matrix of $\hat{\theta}$ only depends on $\xi$ and the variance of $\epsilon_i$, $\sigma^2$, which we assume to be a fixed constant.
In the more general case where the points $t_k$ are not assumed distinct, the covariance matrix is given by
$$\Sigma(\xi)=\sigma^2 \left(\sum_{k=1}^l n_k z(t_k)z(t_k)^T\right)^{-1},$$
provided of course that $\sum_{k=1}^l n_k z(t_k)z(t_k)^T \succ 0$.
In the rest of this section, we consider the case where we would like to minimize the 2-norm of the matrix $\Sigma(\xi)$, which is the same as minimizing its largest eigenvalue. To avoid having to deal with the inverse that appears in its expression however, we will instead maximize the smallest eigenvalue of its inverse, or equivalently, if we define the following quantity,
$$F(\xi)\mathrel{\mathop{:}}=\sum_{k=1}^{l}w_k z(t_k)z(t_k)^T,$$ maximize the minimum eigenvalue of $F(\xi)$. (Note that maximizing the minimum eigenvalue of $F(\xi)$ or of $\Sigma(\xi)^{-1}=\frac{m}{\sigma^2} F(\xi)$ will give the same result.) The matrix $F(\xi)$ is a well-known quantity in statistics called the \emph{Fisher information matrix} of the design matrix $\xi$ and maximizing its minimum eigenvalue corresponds to a common notion of optimality\footnote{There are many different ways of defining optimality, based essentially on minimizing various norms of $\Sigma(\xi)$; we refer the interested reader to \cite{de2017approximate} for more information on this topic.} in experimental design, that of \emph{E-optimality}. Hence, the problem of interest becomes
\begin{align*}
&\max_{\xi} \lambda_{\min} F(\xi)\\
&\text{s.t. } \xi \text{ is as in (\ref{eq:design})}.
\end{align*}
This can be rewritten as 
\begin{align*}
\max_{n_k,t_k,\gamma} &\gamma\\
\text{s.t. } &\sum_{k=1}^l \frac{n_k}{m} z_d(t_k)z_d(t_k)^T \succeq \gamma I\\
&n_k \in \mathbb{N}, \sum_{k=1}^l n_k=m,
\end{align*}
where $I$ is the identity matrix. We first drop the constraint that $n_k \in \mathbb{N}$, relaxing it to $n_k \geq 0$:
\begin{equation}\label{eq:optimal.design.relax1}
\begin{aligned}
\max_{w_k,t_k,\gamma} &\gamma\\
\text{s.t. } &\sum_{k=1}^l w_k z_d(t_k)z_d(t_k)^T \succeq \gamma I\\
&w_k \geq 0, \sum_{k=1}^l w_k=1.
\end{aligned}
\end{equation}
The matrix $\sum_{k=1}^l w_k z_d(t_k)z_d(t_k)^T$ is of size $\mathbb{N}_{d}^n \times \mathbb{N}_{d}^n$. We index it by $(\alpha,\beta)$ where $\alpha, \beta \in \mathbb{N}_{d}^n$. Note that entry $(\alpha,\beta)$ of the matrix is exactly $$\int_{\mathcal{X}} x^{\alpha}x^{\beta}d\delta,$$ where $\delta$ is the Dirac measure given by $\delta(x)=\sum_{k=1}^l w_k \textbf{1}_{x=t_i}(x)$. In other words, entry $(\alpha,\beta)$ of the matrix is the $\alpha+\beta$ moment of $\delta$. Define $$y_{\alpha}\mathrel{\mathop{:}}=\int_{\mathcal{X}} x^{\alpha}d\delta, \alpha \in \mathbb{N}^n_{2d}$$ and let $M(y)$ be the $\mathbb{N}_d^n \times \mathbb{N}_d^n$ matrix with entry $(\alpha,\beta)$ given by $y_{\alpha+\beta}$. It follows that (\ref{eq:optimal.design.relax1}) can be rewritten as:
\begin{equation}\label{eq:optimal.design.relax2}
\begin{aligned}
\max_{w_k,t_k,\gamma,y} &\gamma\\
\text{s.t. } &M(y) \succeq \gamma I\\
&y_0=1, y_{\alpha}=\int_{\mathcal{X}} x^{\alpha}d\delta \text{ for } \alpha \in \mathbb{N}_{2d}^n.
\end{aligned}
\end{equation}
Similarly to (\ref{eq:def.Mnd}), we now define the following set:
\begin{align}\label{eq:def.MdX}
\mathcal{M}_{n,2d}(\mathcal{X}) \mathrel{\mathop{:}}=\{\{y_{\alpha}\}_{\alpha \in \mathbb{N}^n_{2d}}~|~ \exists \text{ a measure } \mu \text{ on $\mathcal{X}$ such that } y_{\alpha}= \int_{\mathcal{X}} x^{\alpha}d\mu, ~\forall \alpha \in \mathbb{N}^n_{2d}\}.
\end{align}
Note that contrarily to (\ref{eq:def.Mnd}), we are considering measures over $\mathcal{X}$ and not over $\mathbb{R}^n$. Hence, the dual of this set is the set of polynomials $p$ with coefficients $\{p_{\alpha}\}_{\alpha \in \mathbb{N}_{2d}^n}$ such that $p$ is nonnegative \emph{over $\mathcal{X}$}, and not over $\mathbb{R}^n$ as previously; see \cite[Section 4.4]{Laurent_survey} for a proof. Using this definition and replacing the specific Dirac measure we were considering by a general measure $\mu$ over $\mathcal{X}$, we are able to further relax (\ref{eq:optimal.design.relax2}) to a problem that only depends on $\gamma$ and $y$:
\begin{equation}\label{eq:optimal.design.relax3}
\begin{aligned}
\max_{\gamma,y}~ &\gamma\\
\text{s.t. } &M(y) \succeq \gamma I\\
&y_0=1, y \in \mathcal{M}_{2d}(\mathcal{X}).
\end{aligned}
\end{equation}

\begin{proposition}
	Consider the optimization problem
	\begin{equation}\label{eq:dual.opt.design}
	\begin{aligned}
	\min_{\lambda \in \mathbb{R}, Q \in \mathbb{N}_{d}^n \times \mathbb{N}_{d}^n}~&\lambda\\
	\text{s.t. } &\lambda-z_d(x)^TQz_d(x) \geq 0, \forall x\in \mathcal{X}\\
	&tr(Q)=1, Q\succeq 0,
	\end{aligned}
	\end{equation}
	where $z_d(x)=(1,x_1,\ldots,x_n,\ldots,x_n^d)$ is the standard vector of monomials. Then weak duality holds between (\ref{eq:optimal.design.relax3}) and (\ref{eq:dual.opt.design}).
\end{proposition}

\begin{proof}
	Let $\lambda,Q$ be feasible for (\ref{eq:dual.opt.design}) and let $\gamma,y$ be feasible for (\ref{eq:optimal.design.relax3}). As $Q\succeq 0$, there exists a matrix $V$ such that $Q=VV^T$. Furthermore, as $tr(Q)=1$, then $tr(VV^T)=tr(V^TV)=1$. Together with the fact that $M(y)\succeq \gamma I$, this implies that $V^TM(y)V\succeq \gamma V^TV$, and in particular, $$tr(V^TM(y)V) \geq tr(\gamma V^TV)=\gamma tr(V^TV)=\gamma.$$
	We have
	\begin{align*}
	\lambda-\gamma \geq \lambda-tr(V^TM(y)V)=\lambda-tr(M(y)Q)=\lambda -\sum_{\alpha,\beta} M_{\alpha,\beta}(y)Q_{\alpha,\beta}
	\end{align*}
	Recall that $M_{\alpha,\beta}(y)=y_{\alpha+\beta}$. As $y$ has a representing measure, it follows that $M_{\alpha,\beta}(y)=\int_{\mathcal{X}} x^{\alpha+\beta} d\mu$. Hence, 
	$$\lambda-\gamma \geq \lambda -\int_{\mathcal{X}} \sum_{\alpha,\beta}Q_{\alpha,\beta}x^{\alpha}x^{\beta} d\mu=\int_{\mathcal{X}} \lambda -z_d^T(x)Qz_d(x) d\mu,$$
	where we have used the fact that $y_0=1$ in the equality. As $\lambda-z_d^T(x)Qz_d(x) \geq 0$, for all $x \in \mathcal{X}$, we deduce that $\lambda-\gamma \geq 0$. 
\end{proof}
Strong duality holds under certain mild conditions, see \cite[Section 4]{Laurent_survey}. As is, (\ref{eq:dual.opt.design}) is not a tractable problem to solve. However, if one replaces the condition that $\lambda-z_d(x)^TQz_d(x)$ be nonnegative over $\mathcal{X}$ by certificates of nonnegativity of the polynomial over $\mathcal{X}$ involving sum of squares polynomials, then the problem becomes a semidefinite program. One can proceed similarly in the primal (\ref{eq:optimal.design.relax3}) by relying on outer-approximations to the set $\mathcal{M}_{2d}(\mathcal{X})$. From the optimal solution $y^*$ to the relaxation of the primal, one can recover---under some assumptions---a measure $\mu$ over $\mathcal{X}$ whose moments correspond to $y^*$. As it turns out, this measure is \emph{atomic} which means that it can be written as a measure over a finite set of points (or \emph{atoms}), each point being associated to a certain weight. We then pick our points $t_k$ to be these atoms, and the corresponding weights to be the atom weights; see \cite{de2017approximate}.

\part{Other applications}

In this last part, we briefly touch upon some other important applications of sum of squares polynomials in optimization and game theory. We wish to stress that, though we have covered a large range of applications in this paper, we have by no means covered all of them. Other important applications of sum of squares techniques which are not included here are for example automated theorem proving \cite{parrilo2004inequality,harrison2007verifying}, extremal combinatorics \cite{raymond2018symmetry}, filer design \cite{genin2000convex,roh2007multidimensional}, and quantum information theory\cite{doherty2004complete,childs2007quantum}. We refer the reader to the aforementioned references if these topics are of interest.

\section{Optimization}
As has already been seen in this volume, as well as briefly in Section \ref{subsec:pop}, sum of squares polynomials are widely used to tackle polynomial optimization problems (POPs), i.e., optimization problems where the objective is a polynomial and the constraints are given by polynomial equalities and inequalities. Though a major application of sum of squares techniques, we won't dwell on POPs as the topic has already been covered. We present however two other areas of optimization, namely copositive programming and robust semidefinite programming, where sum of squares techniques come into play. Multistage optimization is also an area of application of these techniques though not covered here; see \cite{bertsimas2011hierarchy} for more information.

\subsection{Copositive programming} \index{Copositive programming}

\begin{definition}
	An $n \times n$ matrix $M$ is said to be \emph{copositive} if $x^TMx \geq 0$, for all $x\in \mathbb{R}^n_+$. 
\end{definition}
Without loss of generality, it can be assumed that $M$ is symmetric. We denote by $C_n$ the set of $n \times n$ copositive matrices. This is a proper cone in the space of symmetric $n \times n$ matrices. At first glance, this definition may seem similar to the definition of positive semidefiniteness. A major difference between the two is that checking membership of a matrix to the set of positive semidefinite matrices can be done in polynomial time; checking membership of a matrix to $C_n$ however is NP-hard \cite{nonnegativity_NP_hard}.

\begin{proposition}
	Let $$C_n^*\mathrel{\mathop{:}}= \{M \in \mathbb{R}^{n \times n}~|~M=\sum_{i=1}^k y_iy_i^T, y_i \in \mathbb{R}^n, y_i \geq 0, i=1,\ldots,k\}.$$
	This set is the dual cone of $C_n$ for the standard inner product $\langle X,Y \rangle=tr(X^TY)$.
\end{proposition}
We refer to elements of $C_n^*$ as completely positive matrices and to the set $C_n^*$ itself as the completely positive cone. A \emph{copositive program} is none other than the following conic program:
\begin{equation}\label{eq:cp}
\begin{aligned}
&\min_X &&tr(C^TX)\\
&\text{s.t. } &&tr(A_i^TX)=b_i, i=1,\ldots,m\\
& &&X \in C_n,
\end{aligned}
\end{equation}
where $C, A_1,\ldots,A_m$ are $n \times n$ symmetric matrices and $b_1,\ldots,b_m$ are scalars. Its dual problem is given by
\begin{equation}\label{eq:cpp}
\begin{aligned}
&\max_{y \in \mathbb{R}^m} &&\sum b_iy_i\\
&\text{s.t. } &&C-\sum_{i=1}^m y_iA_i \in C_n^*.
\end{aligned}
\end{equation}
A variety of problems can be reformulated as copositive programs or as completely positive programs; see \cite{dur2010copositive} and the references therein for more information on copositive programming. We will see an example relating to the stability number of a graph in Section \ref{subsec:stable.set}.

\subsubsection{Copositive matrices and sum of squares polynomials} As it is NP-hard to check membership to $C_n$, it is of interest to develop sufficient conditions for membership to $C_n$ that are checkable in polynomial time. We give an example of such a condition next.
\begin{proposition}\label{prop:P+N}
Let $M$ be an $n \times n$ symmetric matrix. If $M=P+N$ for some matrices $P \succeq 0$ and $N \geq 0$, then $M$ is copositive.
\end{proposition}
Clearly checking whether $M$ satisfies this condition is a semidefinite (feasibility) program. How strong is this sufficient condition? It turns out that for $n< 5$, the set of matrices $\{M~|~M=P+N, P\succeq 0, N\geq 0\}$ exactly coincides with $C_n$. For $n\geq 5$, this is not true anymore, as evidenced by the so-called Horn matrix, a $5 \times 5$ zero-one matrix which is copositive but cannot be written as $P+N$, $P\succeq 0$, $N\geq 0$; see, e.g., \cite{dur2010copositive}. Better than one sufficient condition, however, would be a hierarchy of sufficient conditions, with each level giving rise to an improved inner approximation of $C_n$. This can be obtained by relating the notion of copositivity back to polynomial nonnegativity and then using sum of squares-based approximations. How can we do this in practice? First, it is straightforward to see that $M$ is copositive if and only if the associated quartic form 
\begin{align}\label{eq:def.m}
m(x)\mathrel{\mathop{:}}=\sum_{i,j}m_{ij}x_i^2x_j^2
\end{align}
 is globally nonnegative. This leads to a natural sufficient condition for copositivity: requiring that $m(x)$ in \ref{eq:def.m} be a sum of squares rather than nonnegative. Interestingly enough, this first sufficient condition is equivalent to that given in Proposition \ref{prop:P+N}.
\begin{proposition}\cite[Section 5]{PhD:Parrilo}
	A matrix $M$ can be written as $P+N$, with $P \succeq 0$ and $N\geq 0$ if and only if the associated quartic form $m(x)$ is a sum of squares.	
\end{proposition}
For the proof, we refer the reader to \cite{PhD:Parrilo}.
Furthermore, this particular rewriting of our initial sufficient condition clues us in on how to construct an improving hierarchy of semidefinite-based inner approximations to the set of copositive matrices; see \cite{PhD:Parrilo}. Indeed, let
$$K_r\mathrel{\mathop{:}}=\{M \in \mathbb{R}^{n\times n}~|~ \left(\sum_{i=1}^n x_i^2\right)^r \cdot \left(\sum_{i,j}^n m_{ij}x_i^2x_j^2 \right) \text{ is sos}\}.$$
It is easy to see that $K_0$ is the set $\{M \in \mathbb{R}^{n \times n}~|~M=P+N, P\succeq 0, N \geq 0\}$, that $K_r \subseteq K_{r+1}, \forall r\geq 0$, and that $K_r \subseteq C_n$, $\forall r\geq 0$. Testing membership to $K_r$ is a semidefinite feasibility program, whose size grows with $r$. An important property of these inner approximations is that they get arbitrarily close to $C_n$, i.e., $int(C_n) \subseteq \cup_{r \in \mathbb{N}} K_r$. The latter fact is a consequence of a theorem by Poly\'a \cite{Polya} which states that for any positive definite even\footnote{As a reminder, we say that a form $f$ is even if each of the variables featuring in its individual monomials has an even power.} form $f$, there exists $r \in \mathbb{N}$ such that $f(x) \cdot (\sum_{i=1}^n x_i^2)^r$ has nonnegative coefficients. Indeed, if $M \in int(C_n)$, i.e., $x^TMx >0, \forall x \neq 0$, then $m(x)$ in (\ref{eq:def.m}) is a positive definite even form. From the aforementioned result, there exists $r \in \mathbb{N}$ such that $m(x) \cdot (\sum_{i=1}^n x_i^2)^r$ has nonnegative coefficients. Combining this with the fact that $m(x) \cdot (\sum_{i=1}^n x_i^2)^r$ is even, it follows that there exists an integer $r$ such that $m(x) \cdot (\sum_{i=1}^n x_i^2)^r$ is a sum of squares and so $M \in K_r$.
\begin{remark}
	The inner approximations to $C_n$ that we have defined here can then be used to obtain upper bounds on (\ref{eq:cp}). Likewise, they can also be used to construct outer approximations to $C_n^*$ and so to obtain upper bounds on (\ref{eq:cpp}).
\end{remark}

\subsubsection{Using copositive programming to approximate the stability number of the graph} \label{subsec:stable.set} \index{Stability number}
Let $G=(V,E)$ be a graph on $n$ nodes. A stable set of the graph $G$ is a subset of its nodes, no two of which have an edge between them. The stability number $\alpha(G)$ of $G$ is the size of the largest stable set in $G$. The ability to compute $\alpha(G)$ has applications in scheduling and coding theory among other areas. The issue however is that $\alpha(G)$ is hard to compute, though reasonable upper bounds on it can often be obtained using linear programming or semidefinite programming. One particularly well-known semidefinite programming approximation is due to Lov\'asz \cite{lovasz1979shannon}; the upper bound obtained via this method is denoted by $\vartheta(G)$ and called the \emph{theta number}. It can be shown that $\vartheta(G)$ is always an improvement on the bounds obtained via the standard linear programming relaxation. Consequently, the paper \cite{lovasz1979shannon} proved to be quite influential in showing how useful semidefinite programming could be. In \cite{deKlerk_StableSet_copositive}, the authors propose a copositive-programming formulation of the stability number. 

\begin{proposition}\cite[Corollary 2.4]{deKlerk_StableSet_copositive}
	For any graph $G=(V,E)$ with adjacency matrix $A$, one has 
	\begin{equation*}
	\begin{aligned}
	\alpha(G)=&\min_{\lambda} \lambda\\
	&\text{s.t. } \lambda(I+A)-ee^T \in C_n, 
	\end{aligned}
	\end{equation*}
where $e$ is the $n \times 1$ vector of ones.
\end{proposition}

Though this reformulation does not make the problem any more tractable, it does open the door to using the aforementioned approximations of the copositive cone, optimization over which is tractable. This gives rise to a hierarchy of optimization problems indexed by $r$:
\begin{equation*}
\begin{aligned}
\vartheta_r'(G)\mathrel{\mathop{:}}=&\min_{\lambda} \lambda\\
&\text{s.t. } \lambda(I+A)-ee^T \in K_r.
\end{aligned}
\end{equation*}
Already at level $r=0$, the authors are able to show that $\vartheta_0'(G)$ improves on $\vartheta(G)$. Furthermore, it is shown that if $r \geq \alpha^2(G)$, then $\vartheta_r'(G)$ is equal to the true stability number. We refer the reader to \cite{deKlerk_StableSet_copositive} for more details.

\subsection{Robust semidefinite programming} \index{Robust semidefinite programming}
We consider the following optimization problem:
\begin{equation}\label{eq:robust.sdp}
\begin{aligned}
u_{opt}\mathrel{\mathop{:}}=&\min_{y \in \mathbb{R}^n} c^Ty\\
&\text{s.t. } F(x,y) \succ 0, \forall x \in \mathbb{R}^m \text{ with } G(x)\preceq 0,
\end{aligned}
\end{equation}
where $c$ is an $n \times 1$ vector, $F$ is a $p \times p$ symmetric matrix whose entries depend polynomially on $x$ and affinely on $y$, and $G(x)$ is a $q \times q$ symmetric matrix whose entries depend polynomially on $x$. We say that $F$ and $G$ are \emph{polynomial matrices}. Note that for fixed $x$, (\ref{eq:robust.sdp}) is a semidefinite program. However, unlike other semidefinite programs, we require here that $y$ satisfy $F(x,y) \succ 0$, regardless of the value that $x$ takes within the set $\{x \in \mathbb{R}^n~|~G(x)\preceq 0\}$. We can view $x$ as encoding uncertainty in our problem and we want our solution $y$ to be robust to this uncertainty. Scenarios where one would encounter semidefinite programs with uncertainty are numerous; see \cite{ben1997robust,el1998robust, scherer2006matrix} for some examples.

We will follow the approach described in \cite{scherer2006matrix}. It relies on the concept of \emph{sum of squares matrix} that we define next. 

\begin{definition}
	A $p \times p$ matrix $S(x)$ whose entries are polynomials in $x \in \mathbb{R}^m$ is said to be a \emph{sum of squares (sos) matrix} if there exists a $n \times p$ polynomial matrix $T(x)$ such that 
	$$S(x)=T(x)^TT(x).$$
\end{definition}
An equivalent definition is for the scalar-valued polynomial $y^TS(x)y$ be a sum of squares polynomial in the variables $(x,y) \in \mathbb{R}^m \times \mathbb{R}^p$. As a consequence, constraining a matrix to be a sum of squares matrix can be imposed using semidefinite programming.

\begin{remark}
	This may seem similar to the concept of sos-convexity, which we saw in Section \ref{sec:shape.constr.reg}. In fact, there is a link between both notions as a polynomial is sos-convex if and only if its Hessian is an sos matrix. 
\end{remark} 

The difficulty in solving (\ref{eq:robust.sdp}) lies in the ``for all $x$'' quantifier. A key step is then to reformulate the constraint in such a way that the quantifier does not appear anymore. This is done by making use of a bilinear mapping
$$(\cdot,\cdot)_p:\mathbb{R}^{pq \times pq}\times \mathbb{R}^{q \times q} \mapsto \mathbb{R}^{p \times p},$$
such that $$(A,B)_p=tr_p(A^T (I_p \otimes B)),$$ where $\otimes$ is the Kronecker product and 
$$tr_p(C)=\begin{bmatrix} tr(C_{11}) & \ldots & tr(C_{1p})\\ \vdots \ddots \vdots \\ tr(C_{p1}) & \ldots & tr(C_{pp}) \end{bmatrix} \text{ with } C \in \mathbb{R}^{pq \times pq}, C_{jk} \in \mathbb{R}^{q\times q}.$$ 
\begin{proposition}\label{prop:robust.sdp}
	Let $A \in \mathbb{R}^{pq \times pq}$ and $B \in \mathbb{R}^{q \times q}$. If $A \succeq 0$ and $B \succeq 0$, then $(A,B)_p \succeq 0$.
\end{proposition}
\begin{proof}
	As $B \succeq 0$, there exists a matrix $D \succeq 0$ such that $B=DD^T$. As $A \succeq 0$, it follows that 
	$$(I_p \otimes D)^TA(I_p \otimes D) \succeq 0.$$ 
	If we denote by $\{A_{ij}\}$ the $q \times q$ blocks that make up $A$ and by $\{C_{ij}\}$ the $q \times q$ blocks that make up $C\mathrel{\mathop{:}}=(I_p \otimes D)^TA(I_p \otimes D)$, then it follows that
	$$C_{ij}=0 \text{ if } i \neq j \text{ and } C_{ii}=D^TA_{ii}D.$$
	As $A_{ii} \succeq 0$ (this follows from the fact that $A \succeq 0$), then $tr_P(C) \succeq 0$. Manipulations of the Kronecker product then give us
	$$tr_P((I_p \otimes D)^TA(I_p \otimes D))=tr_P(A(I_p \otimes D)(I_p \otimes D))=tr_P(A(I_p \otimes DD^T))=tr_P(A(I_p \otimes B)),$$
	which enables us to conclude that $(A,B)_p \succeq 0$.
\end{proof}
Going back to our robust semidefinite programming problem in (\ref{eq:robust.sdp}), we now define a new optimization problem:
\begin{equation}\label{eq:robust.sdp2}
\begin{aligned}
v_{opt}\mathrel{\mathop{:}}=&\min_{y,\epsilon,S} &&c^Ty\\
&\text{s.t. } &&\epsilon>0,\\
& &&F(x,y)+(S(x),G(x))_p -\epsilon I_p  \text{ is an sos matrix},\\
& &&S(x) \text{ is an sos matrix}.
\end{aligned}
\end{equation}
Here $S(x)$ is a $pq \times pq$ matrix with polynomial entries. We have the following result.

\begin{proposition}\cite{scherer2006matrix}
	Let $u_{opt}$ and $v_{opt}$ be defined as in (\ref{eq:robust.sdp}) and (\ref{eq:robust.sdp2}). We have $u_{opt} \leq v_{opt}$.
\end{proposition}
\begin{proof}
	Suppose $y, \epsilon$, and $S(x)$ are feasible for (\ref{eq:robust.sdp2}) and let $x_0$ be an arbitrary point such that $G(x_0) \preceq 0$. Using the third constraint of (\ref{eq:robust.sdp2}), we have that $S(x_0) \succeq 0$. From the second constraint in (\ref{eq:robust.sdp2}) and Proposition \ref{prop:robust.sdp}, it follows that
	$$F(x_0,y)-\epsilon I_p \succeq -(S(x_0),G(x_0))_p=(S(x_0),-G(x_0))_p \succeq 0$$
	and so $F(x_0,y) \succ 0$. This implies that $y$ is feasible for (\ref{eq:robust.sdp2}) as $x_0$ was chosen arbitrarily, and so that $u_{opt} \leq v_{opt}$.
\end{proof}

Under a modest technical condition on $G(x)$ (a generalization of the so-called Archimedean condition to matrices), it can in fact be shown that $u_{opt}=v_{opt}$; see \cite[Theorem 1]{scherer2006matrix}. Though (\ref{eq:robust.sdp2}) is not tractable as is, it can easily be made tractable by fixing the degree of the polynomials that appear in the entries of $S(x)$. Indeed, it then becomes a semidefinite program. Under the previous assumption, and as the degree of the polynomials in $S$ grows, the optimal values of the approximations converge to $u_{opt}$. Once again, we refer the reader to \cite{scherer2006matrix} for more details.

\section{Game theory} \index{Game theory}

In a game, multiple agents (the \emph{players}) have to make decisions, each guided by his or her preferences, and each influenced by the decisions of the other players. Game theory is the mathematical theory that studies what behavior emerges in these situations; see, e.g., \cite{osborne1994course}. In this section, we consider \emph{strategic form games} which are games specified by three components: the number of players, the set of actions (or pure strategies) for each player, and their payoffs. Both of the games we consider have two players, Player 1 and Player 2, each having a set of actions $S_1$ and $S_2$, and a payoff function $P_i: S_1 \times S_2 \rightarrow \mathbb{R}$, $i=1,2$. The sets $S_i,~i=1,2$ can be finite or infinite. We further assume that the games in question are zero-sum games, i.e., $P_1=-P_2$. In strategic games, when player $i$ plays, (s)he can decide to play one action in a deterministic fashion---a \emph{pure} strategy---or (s)he can decide to pick a strategy that doesn't assign all mass to one strategy, but rather a probability distribution over the set of strategies $S_i$. The latter is called a \emph{mixed} strategy. We assume that each player randomizes their strategy independently of the other. We consider two types of games here: polynomial games (Section \ref{sec:poly.games}) and polynomial stochastic games (Section \ref{sec:stoch.poly.games}). For each, our goal is to compute---in an efficient manner---mixed strategies for each player, that are \emph{solutions} to the game. In both cases, the solution concept we consider is a Nash equilbrium. We will specify in each section exactly what is meant by this as the nature of the game changes.

\subsection{Polynomial games}\label{sec:poly.games} \index{Polynomial games}
We consider a game where $S_1=[-1,1]$, $S_2=[-1,1]$ and the payoff function of Player~1 is a polynomial in actions $x\in S_1$ and $y \in S_2$:
$$P_1(x,y)=\sum_{i=0}^n  \sum_{j=0}^m p_{ij}x^iy^j.$$
Similarly, the payoff function of Player~2 is a polynomial in actions $x 
\in S_1$ and $y \in S_2$.
As mentioned previously, our goal is to efficiently compute a pair of mixed strategies, also called a \emph{strategy profile}, that is a Nash equilibrium. We use $\Lambda(S)$ to denote the set of probability measures over $S$ and by $\mu \times \nu$ the product measure of two measures $\mu$ and $\nu$.
\begin{definition} \label{def:Nash}
Consider a two-player polynomial game with payoff functions $P_1(x,y),P_2(x,y)$. Let $\mu^*$ (resp. $\nu^*$) be a probability measure over $S_1$ (resp. $S_2$). A mixed strategy profile $(\mu^*, \nu^*)$ is a Nash equilibrium if
	$$E_{\mu^* \times \nu^*}[P_1(x,y)] \geq E_{\mu \times \nu^*}[P_1(x,y)], \forall \mu \in \Lambda(S_1)~~\text{ and }~~E_{\mu^* \times \nu^*}[P_2(x,y)] \geq E_{\mu^* \times \nu}[P_2(x,y)], \forall \nu \in \Lambda(S_2).$$	
\end{definition}
In layman's terms, this means that if one player changes his strategy while the strategy of the other player remains the same, his or her expected payoff cannot increase. Thus, there is no incentive for any one player to unilaterally change his or her strategy.

\begin{remark}
	In general, Nash equilibria do not necessarily exist, though there are a number of games for which they are known to. For example, Nash himself showed that there always exists a Nash equilibrium for games that involve a finite number of players and a finite number of pure strategies. Existence of these equilibria does not guarantee however that they are easy to compute: in fact, it is believed that computing Nash equilibria is a hard problem, even in the simple  2-player finite-strategy setting; see \cite{chen2006settling,daskalakis2009note}. There are a few specific cases where Nash equilbria computation can be done in polynomial time---when a 2-player finite-strategy game is zero-sum for example \cite{dantzig1951proof}. In the two cases considered below, namely polynomial games and polynomial stochastic games, we will see that not only do Nash equilibria exist, but they can be computed using semidefinite programming. 
\end{remark}

We further assume that our game is zero sum, i.e., $P_1=-P_2$. For ease of notation, we consequently drop the subscript and write $P$ instead of $P_1$ (and $-P$ instead of $P_2$). This is the simplest case of a polynomial game, first introduced in \cite{dresher2016polynomial}. From the zero-sum game assumption, the following equivalence is immediate.
\begin{proposition}
	Consider a zero-sum polynomial game with payoff function $P(x,y)$. Let $\mu^*$ (resp. $\nu^*$) be a probability measure over $S_1$ (resp. $S_2$) The strategy profile $(\mu^*,\nu^*)$ is a Nash equilibrium if and only if $(\mu^*,\nu^*)$ satisfies the saddle-point condition
	\begin{align}\label{eq:saddle.point}
	E_{\mu \times \nu^*}[P(x,y)] \leq E_{\mu^* \times \nu^*}[P(x,y)] \leq E_{\mu^* \times \nu}[P(x,y)],~~\forall \mu \in \Lambda(S_1), \forall \nu \in \Lambda(S_2).
	\end{align}
\end{proposition}

We denote by $\mathcal{M}_{1}([-1,1])$ the set of moments which have a representing measure over $[-1,1]$. This is similar to the notation employed in Section \ref{sec:moment.pb} for the global setting. The main theorems in \cite{Pablo_poly_games}, from which this section is inspired, state the following. 

\begin{theorem}\cite[Theorem 2.2-Theorem 2.7]{Pablo_poly_games}
	Consider a zero-sum polynomial game with payoff function $P(x,y)$. There exist mixed strategies $(\mu^*,\nu^*)$ verifying the saddle-point condition (\ref{eq:saddle.point}). Furthermore, these strategies can be obtained by solving the primal-dual pair of optimization problems:
	\begin{equation} \label{eq:primal.dual.pair}
	\begin{aligned}
	&\max_{\mu_0,\ldots,\mu_n, \lambda} &&\lambda \qquad &&&\min_{\nu_0,\ldots,\nu_m, \gamma} &&&& \gamma   \\
	&\text{s.t. } &&\sum_{j=0}^m (\sum_{i=0}^n p_{ij}\mu_i)y^j -\lambda \geq 0, \forall y \in [-1,1] \qquad \text{ and } &&&\text{s.t. } &&&&\gamma - \sum_{i=0}^n (\sum_{j=0}^m p_{ij}\nu_j)x^i \geq 0, \forall x \in [-1,1]\\
	& &&(\mu_0,\ldots,\mu_n) \in \mathcal{M}_{1}([-1,1]),~\mu_0=1 &&&   &&&&(\nu_0,\ldots,\nu_m) \in \mathcal{M}_{1}([-1,1]),~\nu_0=1
	\end{aligned}
	\end{equation}
	\end{theorem}

\begin{remark}
In the previous optimization problems, $(\mu_0,\ldots,\mu_n)$ and $(\nu_0,\ldots,\nu_m)$ are moments of a measure over $[-1,1]$. As the measures considered are over intervals and the polynomials are univariate, the constraints of these problems have a tractable semidefinite representation. We refer the reader to \cite{Pablo_poly_games} for the exact semidefinite programming formulation. From the moments  $(\mu_0,\ldots,\mu_n)$ and $(\nu_0,\ldots,\nu_m)$ obtained, one can extract corresponding \emph{atomic} measures $\mu^*,\nu^*$ over $[-1,1]$, which are measures over a finite number of points. These measures constitute our desired strategies. Again, the extraction process is not covered here but can be found, e.g., in \cite[Section 3.3.5]{PabloGregRekha_BOOK}.
\end{remark}

This theorem gives us both the existence of a Nash equilibrium and the means to compute it efficiently. Before we proceed, we give some intuition as to how these optimization problems arise. It is straightforward to show that for the game at hand, $(\mu^*,\nu^*)$ satisfies the saddle-point condition (\ref{eq:saddle.point}) if and only if 
\begin{align} \label{eq:maximin.th}
\max_{\mu \in \Lambda(S_1)} \min_{\nu \in \Lambda(S_2)} E_{\mu \times \nu} [P(x,y)]= \min_{\nu \in \Lambda(S_2)} \max_{\mu \in \Lambda(S_1)} E_{\mu \times \nu} [P(x,y)].
\end{align}
and $$\mu^* \in \arg \max_{\mu \in \Lambda(S_1)} \min_{\nu \in \Lambda(S_2)} E_{\mu \times \nu} [P(x,y)],\quad \nu^* \in \arg \min_{\nu \in \Lambda(S_2)} \max_{\mu \in \Lambda(S_1)} E_{\mu \times \nu} [P(x,y)].$$
We call the common optimal value the \emph{value} of the game, $\mu^*$ a \emph{maximin} strategy and $\nu^*$ a \emph{minimax} strategy. These strategies are intuitive: as $P(x,y)$ is the payment from Player 2 to Player 1, it follows that Player 1 would try to maximize the payoff in the worst case scenario that Player 2 has played a strategy that minimizes it as much as possible, and conversely. How does this equivalent characterization help us obtain the optimization problems in (\ref{eq:primal.dual.pair})? It suffices to simply rewrite the problems that appear in (\ref{eq:maximin.th}). Consider, for example,
\begin{equation}\label{eq:maximin}
\begin{aligned}
\max_{\mu \in \mu(S_1)} \min_{\nu \in \mu(S_2)} E_{\mu \times \nu} [P(x,y)].
\end{aligned}
\end{equation}
We use the fact that $P(x,y)$ is a polynomial in $x$ and $y$ to express the expectation as a polynomial in its moments. This, in addition to the fact that the strategies are randomized independently, allows us to conclude
$$E_{\mu \times \nu}[P(x,y)]=\sum_{i=0}^n \sum_{j=0}^m p_{ij} E_{\mu \times \nu}[x^iy^j]=\sum_{i=0}^n \sum_{j=0}^m p_{ij} \mu_i \nu_j$$ 
where $$\mu_i=\int_{-1}^1 x^i d\mu(x) \text{ and }\nu_j=\int_{-1}^1 y^j d\nu(y).$$
As a consequence, (\ref{eq:maximin}) can be written as:
\begin{equation*}
\begin{aligned}
&\max_{(\mu_0,\ldots,\mu_n) \in \mathcal{M}_{1,n}([-1,1])} \begin{bmatrix}\min_{(\nu_0,\ldots,\nu_{m}) \in \mathcal{M}_{1,m}([-1,1])} \sum_{i=0}^n \sum_{j=0}^m p_{ij} \mu_i \nu_j \\ \text{s.t. } \nu_0=1 \end{bmatrix}\\
&\text{s.t. } \mu_0=1,
\end{aligned}
\end{equation*}
where $\mathcal{M}_{1,n}([-1,1])$ is once again defined as the set of (truncated) moment sequences that have a representative measure over $[-1,1]$. We now use duality to rewrite the inner problem: note that this problem is very similar to (\ref{eq:pop.primal}) and so its dual will be of a similar form to (\ref{eq:pop.dual}), the only difference being that the measures are over $[-1,1]$. Problem (\ref{eq:maximin}) then becomes:
\begin{equation*}
\begin{aligned}
&\max_{\mu_i, \lambda} &&\lambda\\
&\text{s.t. } &&\sum_{j=0}^m (\sum_{i=0}^n p_{ij}\mu_i)y^j -\lambda \geq 0, \forall y \in [-1,1]\\
& &&(\mu_0,\ldots,\mu_n) \in \mathcal{M}_{1,n}([-1,1]),~\mu_0=1,
\end{aligned}
\end{equation*}
which is the first of the two problems in (\ref{eq:primal.dual.pair}). One can proceed likewise for the second problem.

\begin{example}\label{ex:tinder}
Consider a couple, Danny and Sandy, who have just gone on their first date together. At the end of the date, they agree to message each other within the next two days, which we model by the interval $[-1,1]$. For both of them, the decision of when to send their message is governed by two principles: not appearing too keen but still signaling interest. As a consequence, some scenarios are preferable to others: it is better, e.g., to send a message after the other has sent his or hers rather than before; but too much time shouldn't go by before they contact each other (and if one of them has to send a message before the other, then it is better that it be closer to the end of the date so it can be construed as a polite thank-you message rather than over-keenness).

We model this situation by the following two-player zero-sum polynomial game. Let $x \in [-1,1]$ be the moment at which Danny sends his/her message and $y \in [-1,1]$ be the moment at which Sandy sends his/her message. The payoff function for Danny is given by
\begin{align*}
P(x,y)=-\frac{(x-y)^3}{3}+x-y.
\end{align*}
To observe the behavior described in the previous paragraph, consider Figure \ref{fig:trash}. We have plotted Danny's payoff for different strategies of Sandy's. For example, if Sandy messages Danny right after the date ($y=-1$), then Danny's payoff is maximum if he/she responds one day later. Likewise, if Sandy messages Danny a day after the date ($y=0$), Danny is best off responding two days later. Finally, if Sandy messages Danny a full two days after the date ($y=1$), then Danny is better off messaging Sandy immediately after the date so as to not let too much time ellapse between the date and the first message between them.

%
	
	\begin{figure}[h!]
		\centering
		\subfigure[$y=-1$]{\includegraphics[width = 0.32\textwidth]{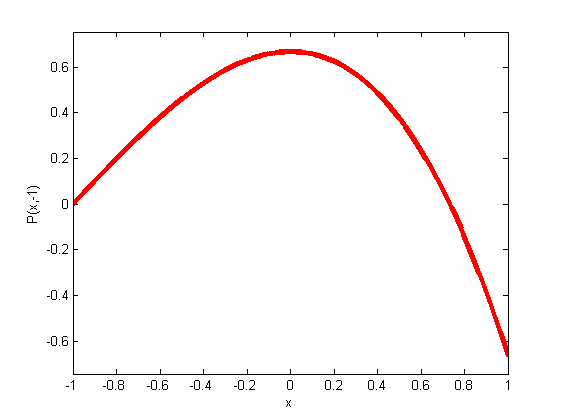}}
		\subfigure[$y=0$]{\includegraphics[width = 0.32\textwidth]{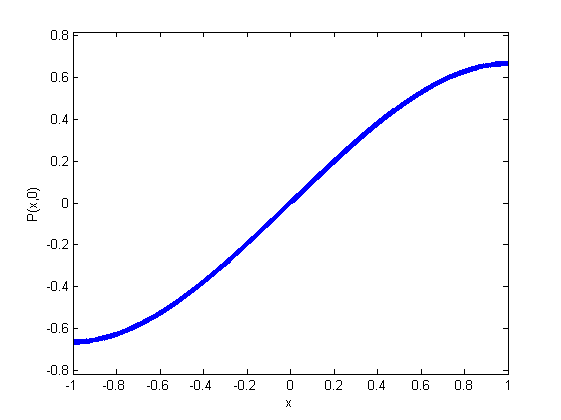}} 
			\subfigure[$y=1$]{\includegraphics[width = 0.32\textwidth]{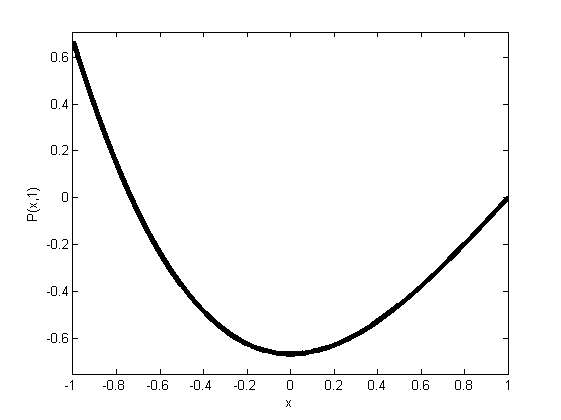}} 
		\caption{Plots of $P(x,y)$ for the game described in Example \ref{ex:tinder} for different values of $y$}
		\label{fig:trash}
	\end{figure}
	
	To find the solution strategies $(\mu^*,\nu^*)$ to this game, we solve the semidefinite programs given in (\ref{eq:primal.dual.pair}) using YALMIP \cite{yalmip} and MATLAB. For the moments of $\mu^*$, we obtain 
	$$(\mu_0^*,\mu_1^*,\mu_2^*,\mu_3^*)=(1,0.1547,0.6906,-0.0718).$$
	This gives rise to the following strategy for Danny:
	message Sandy immediately after the date $33.34\%$ of the time; or message him/her at time $x=0.7321$, $66.66\%$ of the time.
	Likewise, for the moments of $\nu^*$, we obtain
	$$(\nu_0^*,\nu_1^*,\nu_2^*,\nu_3^*)=(1,-0.223,0.7407,-0.3951).$$
	This gives rise to the following strategy for Sandy:
	message Danny immediately after the date $53.34\%$ of the time; or message him/her at time $x=2/3$, $46.66\%$ of the time.
	These two strategies are aligned with the incentives of both players. They either go for the strategy of messaging right after the date: this guarantees that their first communication is not too far from the end of their date but reduces the satisfaction of getting a message before sending one. Or, they message much closer to the end of the two days in the hope that the other person will have messaged them before, though possibly running the risk of seeming uninterested.
%

\end{example} 

To conclude, in the case of two-player zero-sum polynomial games, computing a Nash equilibrium can be done by solving a pair of primal-dual semidefinite programs. This is a nice generalization of two-player zero-sum finite games where computing a Nash equilibrium amounts to solving a pair of primal-dual linear programs. An extension of polynomial games to the multivariate setting can be found in \cite{laraki2012semidefinite}; other extensions to, e.g., different notions of equilibria such as correlated equilibria \cite{stein2011correlated} also exist.

\subsection{Polynomial stochastic games}\label{sec:stoch.poly.games} \index{Polynomial stochastic games}

We consider an interesting extension---called a \emph{polynomial stochastic game} \cite{shah2007polynomial}---of the polynomial game seen in the previous subsection. At a high level, a polynomial stochastic game can be viewed as a repetition of many polynomial games, where the transition from one game to another is dictated by the current game. More specifically, the game is again between two players and is as played as follows. We are given finitely many states $\mathcal{T}=\{1,2,\ldots,T\}$ and a set of actions $S_1$ and $S_2$ for each player. For simplicity, we assume that $S_1=S_2=[0,1]$ but the results in \cite{shah2007polynomial} can be extended to any finite union of intervals on the real line. The players start in an initial state $t$ and move to the next state depending on their current state and the actions of Player 1 \emph{only}. This deviates from the standard stochastic game where the actions of both players play a role in the choice of the next state: this assumption is called the \emph{single-controller assumption}. We define $\pi(t';t,x)$ to be the probability of transitioning from state $t$ to state $t'$ given that Player 1 has taken action $x$ within $S_1$. 

We assume that the game runs over an infinite horizon generating sequences of states $(t_1,\ldots,t_k,\ldots)$ and sequences of actions $(x_1,\ldots,x_k,\ldots)$ and $(y_1,\ldots,y_k,\ldots)$ from both players. In this context, the payoff of Player 1 for the whole game is given by:
$$\sum_{k=1}^\infty \beta^k P(t_k,x_k,y_k),$$
where $\beta<1$ is a discount factor and $P(t_k,x_k,y_k)$ is the payoff received by Player 1 when the players are in state $t_k$ and actions $x_k \in S_1$ and $y_k \in S_2$ are played. As the game is polynomial, we assume that the payoff of Player 1 is a polynomial in the actions $x \in S_1$ of Player 1 and $y \in S_2$ of Player 2, i.e.,
$$P(t,x,y)=\sum_{i=0}^n \sum_{j=0}^m p_{ij}(t) x^i y^j.$$
Here, $p_{ij}(t)$ is a polynomial in $t$, with $n$ being the degree of $P$ in $x$, and $m$ in $y$. As the game is zero sum, the payoff of Player 2 is simply the negative of the payoff of Player 1. We will also assume that the transition probability $\pi(t';t,x)$ is polynomial in $x$ for every fixed integers $t$ and $t'$:
$$\pi(t',t,x)=\sum_{i=0}^n \pi_{i,t',t} x^i.$$

The full set of strategies for this game is very cumbersome: indeed, the strategy that should be played at time $k$ depends a priori on the past sequence of states and actions as well as the current state. As a consequence, we will restrict ourselves to searching for optimal \emph{stationary} strategies, which are strategies that only depend on the players' current state. This means that instead of searching over an infinite set of probability distributions, we are able to search over a finite set of probability distributions. (We will see that this restriction doesn't lead to any loss as an optimal stationary distribution always exists within the solution concept that we define.) More specifically, we wish to find, for Player 1, a vector of (mixed) strategies $M\mathrel{\mathop{:}}=(\mu^1,\ldots,\mu^T)$, where $\mu^i$ is a probability distribution over $S_1$ for $i=1,\ldots,T$ and, for Player 2, a vector of (mixed) strategies $N\mathrel{\mathop{:}}=(\nu^1,\ldots,\nu^T)$, where $\nu^i$ is a probability distribution over $S_2$ for $i=1,\ldots,T$. Given the stationarity assumption, we can rewrite the payoff of Player 1 for the whole game in a simpler way. To do this, we introduce a $T \times T$ probability transition matrix which depends on Player's 1 strategy only as previously mentioned:
$$\Pi_{tt'}(M)=\int_{S_1} \pi(t',t,x) d\mu^t(x).$$
Furthermore, the average payoff for Player 1 obtained for being at state $t$ with Player 1 playing strategy $\mu^t$ and Player 2 playing strategy $\nu^t$ is given by
$$E_{\mu^t \times \nu^t}[P(s,x,y)]=\int_{S_1} \int_{S_2} P(t,x,y)d\mu^t(x) d \nu^t(y).$$
It follows then that the average overall payoff $v_{\beta}(t,M,N)$ for Player 1, for fixed strategies $M$, $N$, discount factor $\beta$, and starting position $t$ is the solution of the set of equations:
$$v_{\beta}(t,M,N)=E_{\mu^t \times \nu^t}[P(t,x,y)]+\beta \sum_{t' \in \mathcal{T}} \Pi_{tt'} \cdot v_{\beta}(t',M,N),~ t=1,\ldots,T.$$
The average overall payoff for Player 2 for a game starting at $t$ is given by $-v_{\beta}(t,M,N)$ as the game is zero-sum.
Once again, we would like to compute solutions $(M^*,N^*)$ to the game that are a Nash equilibrium. Our definition of a Nash equilibrium is slightly different to the one given above as we are considering vectors or probability distributions instead of one probability distribution. We use $S_1^T$ to denote the set $\underbrace{S_1 \times S_1 \times \ldots S_1}_{T \text{ times}}$ and similarly for $S_2$.

\begin{definition}
	Consider a two-player zero-sum polynomial stochastic game with states $(1,\ldots,T)$, sets of actions $S_1$ and $S_2$, discount factor $\beta$, transition probability $\pi(t',t',x)$, and payoff function $P(t,x,y)$. Let $M^*$ (resp. $N^*$) be vectors of probability measures of length $T$ over $S_1$ (resp. $S_2$). The pair $(M^*, N^*)$ is a Nash equilibrium if
	\begin{align}\label{eq:saddle.point.stoch}
	v_{\beta}(t,M,N^*) \leq v_{\beta}(t,M^*,N^*) \leq v_{\beta}(t,M^*,N), \forall t=1,\ldots,T,~\forall M \in \Lambda(S_1^T),N \in \Lambda(S_2^T).
	\end{align}
	\end{definition}
The characterization of a Nash equilibrium that we have given is in fact the saddle-point condition. This can be immediately translated to the condition that 
$$v_{\beta}(t,M,N^*) \leq v_{\beta}(t,M^*,N^*) \text{ and } -v_{\beta}(t,M^*,N^*) \geq -v_{\beta}(t,M^*,N)$$ which is of similar flavor to the definition of a Nash equilibrium that we gave in the previous subsection. The latter condition can be interpreted as: regardless of the state at which the game starts, there is no incentive for one player to deviate from his or her strategy if the other player does not.

We denote by $\nu_j^t$ (resp. $\mu_i^{t}$) the $j^{th}$ (resp. $i^{th}$) moment of $\nu^t$ (resp. $\mu^t$). The main results in \cite{shah2007polynomial} can be expressed as follows.

\begin{theorem}
Consider a two-player zero-sum polynomial stochastic game with states $(1,\ldots,T)$, sets of actions $S_1$ and $S_2$, discount factor $\beta$, transition probability $\pi(t',t',x)$, and payoff function $P(t,x,y)$. There exist optimal vectors of mixed strategies $(M^*,N^*)$ verifying the saddle-point condition (\ref{eq:saddle.point.stoch}). Furthermore, these optimal strategies can be obtained by solving the primal-dual pair of optimization problems:
	\begin{equation}\label{eq:stoch.primal}
	\begin{aligned}
	&\min_{v \in \mathbb{R}^T, \{(\nu^t_0,\ldots,\nu^t_m),t=1,\ldots,T\}} &&\sum_{t=1}^T v_t\\
	&\text{s.t. } && v_t \geq \sum_{i=0}^n \sum_{j=0}^m p_{ij}(t) \nu_j(t) x^i +\beta \sum_{t' \in \mathcal{T}} \left(\sum_{i=0}^n \pi_{i,t,t'} x^i\right)v_t', \forall x \in [0,1], \forall t=1,\ldots,T\\
	& &&(\nu_0^t,\ldots,\nu_m^t) \in \mathcal{M}_{1}([0,1]), \nu_0^t=1, \forall t=1,\ldots,T.
	\end{aligned}
	\end{equation}
		\begin{equation} \label{eq:stoch.dual}
	\begin{aligned}
	&\max_{\alpha \in \mathbb{R}^T, \{(\xi_0^t,\ldots,\xi_n^t),t=1,\ldots,T\}} &&\sum_{t=1}^T \alpha_t\\
	&\text{s.t. } && \sum_{i=0}^n \sum_{j=0}^m p_{ij}(t) \xi_i^t y^j \geq \alpha_t,\forall y \in [0,1], \forall t=1,\ldots,T,\\
	& &&(\xi_0^t,\ldots,\xi_n^t) \in \mathcal{M}_{1}([0,1]), \forall t=1,\ldots,T,\\
	& &&\sum_{t=1}^T \left( \xi_0^t-\beta \sum_{i=0}^n \pi_{i,t,t'} \xi_i^t \right) =1, \forall t'=1,\ldots,T.
	\end{aligned}
	\end{equation}
\end{theorem}
We refer the reader to \cite{shah2007polynomial} for a proof of why these optimization problems give rise to mixed strategies that verify the saddle point-condition. 
It is easy to see that both problems can be rewritten as semidefinite programs as the measures involved are over an interval and the polynomials involved are univariate. To obtain the optimal strategy $N^*$ from (\ref{eq:stoch.primal}), it suffices to recover the representing measure $\nu^{t*}$ from the optimal moments $(\nu^{t*}_0,\ldots, \nu^{t*}_m)$ for all $t$; see \cite[Section 3.3]{PabloGregRekha_BOOK}. For $M^*$, the moments $(\xi^{t*}_0,\ldots,\xi^{t*}_n)$ first need to be normalized so as to correspond to a probability measure: this just involves dividing by $\xi^{t*}_0$. Then the same method that is applied to recover $N^*$ can be used to recover $M^*$.

\part{Conclusion: a word on implementation challenges} \index{Parsers} \index{Solvers} \index{Scalability challenges}

A topic that is central to applications but that we have barely touched upon so far is how to implement these methods in practice. In particular, what software should we use to solve sum of squares programs? There are two components here to consider: which semidefinite programming solver to use and which parser to use which converts the sum of squares program to a semidefinite program. Indeed, in theory, one could manually convert the sum of squares program at hand into a semidefinite program, but this can be quite tedious. It is consequently much more convenient to use a parser whose role is to automate this process. Note however that not all parsers interface with all solvers, and that one need sometimes access parsers or solvers within other software or interfaces (e.g., MATLAB). 

We start by reviewing a few semidefinite programming solvers (the list is by no means meant to be exhaustive). Choosing a good-quality solver is a crucial step in coming up with robust solutions to the problem at hand in a reasonable amount of time. MOSEK \cite{mosek}, SDPT3 \cite{SDPT3}, Sedumi \cite{sedumi}, and SDPA \cite{yamashita2003implementation} are established solvers, with the first being a commercial solver (free with an academic license), and the latter three being free. SDPA interfaces with Python, MOSEK interfaces with C, Java, Python, and MATLAB, and Sedumi and SDPT3 interface with MATLAB. All of these solvers rely on interior point methods, which unfortunately do not always scale very well with the size of the problem. In light of this, new solvers have been developed which rely on augmented Lagrangian methods instead, such as SDPNAL/SDPNAL+ \cite{yang2015sdpnal}, CDCS \cite{zheng2017chordal}, and SCS \cite{o2016conic}. These are all free. The first uses Newton-Conjugate Gradient augmented Lagrangian, whereas the latter two use ADMM. Furthermore SDPNAL/SDPNAL+ and CDCS can be accessed from MATLAB, whereas SCS is written in C and can be used in other C, C++, Python environments, as well as in MATLAB, R and Julia. 

In terms of parsers, the most commonly used are perhaps YALMIP \cite{yalmip}, Julia \cite{bezanson2017julia}, SOSTOOLS \cite{sostools}, SPOT \cite{SPOT_Megretski}, Gloptipoly \cite{henrion2009gloptipoly}, and Macaulay2 \cite{M2}. YALMIP is a toolbox for modelling and optimization in MATLAB that has a special sos module. It can be interfaced with many solvers including MOSEK, SDPT3 and Sedumi. Julia is an open-source programming language for technical computing. Optimization is done via its modeling package JuMP \cite{dunning2017jump} and sum of squares problems can be tackled via one of its packages SumofSquares.jl \cite{brackston2018construction}. SOSTOOLS is a free MATLAB toolbox for sos programs. It can be interfaced with a number of solvers including SeDuMi, SDPT3, CSDP, SDPNAL, SDPNAL+, CDCS and SDPA. SPOT can be viewed as an alternative to SOSTOOLS as it is also a MATLAB toolbox for sos programs, but its focus is mainly towards control theory. GloptiPoly is a free MATLAB toolbox, which solves what is known as the \emph{generalized moment problem}. This includes the moment problem as described in Section \ref{sec:moment.pb}, and hence can be used to tackle polynomial optimization problems. It can be interfaced with many different solvers including Sedumi, SDPT3 and MOSEK. Finally, Macaulay2 is a free computer algebra system geared towards research in algebraic geometry. However, via a package, it can be used to solve sum of squares programs. 

As mentioned above, a direction currently taken in solver development involves replacing interior point methods by methods that can robustly solve very large problems, such as ADMM \cite{boyd2011distributed}. This is due to the fact that the semidefinite programs generated by sum of squares programs are very large: as an example, if the polynomials considered in an sos program are of degree $2d$ and in $n$ variables, then the associated semidefinite program will be of order $n^d$. This limited ability to solve very large sos programs has been one of the main impediments in further disseminating sum of squares techniques. Indeed, possible new applications often feature problems of large scale. This has consequently led to a flurry of research around the question: how can we make solving sos programs more scalable? One such step of course is to construct new solvers for semidefinite programs that rely on more scalable algorithms as we saw above. Another research direction, complementary to this one, is to leverage the structure of the semidefinite program at hand to reduce its size. Structures of interest can include e.g. symmetries in the problem or sparsity; see \cite{Symmetry_groups_Gatermann_Pablo,vallentin2009symmetry,waki2006sums} for some of these directions. A very different research direction involves replacing the semidefinite program at hand by cheaper conic programs with trade-offs in accuracy; see \cite{ahmadi2017dsos,weisser2018sparse} for some examples of this direction. The hope is that by combining these different research directions, one will be able to tackle large-scale sos programs and open up many new areas to the use of sos programming.

\section*{Acknowledgments} 
The author would like to thank Amir Ali Ahmadi for painstakingly going over an initial version of this manuscript to provide the author with feedback. His very relevant comments considerably improved this draft. The author would also like to thank Jeffrey Zhang for the example in equation (\ref{eq:CE.Jeff}).

\bibliographystyle{amsalpha}
\bibliography{ams_lecture_notes}
%
%
%
%

\end{document}